\documentclass{article}

\usepackage{hhline}
\usepackage{longtable}
\usepackage{easyReview}
\usepackage[margin=1.2in]{geometry}
\usepackage{mathtools}
\usepackage[affil-sl]{authblk}

\usepackage{cla}
\crefrangeformat{equation}{Equations~(#3#1#4)--(#5#2#6)}

\title{Representations of compatible Lie algebras}

\author[1]{Xabier García-Martínez \thanks{xabier.garcia@usc.gal}}
\author[1]{Manuel Ladra \thanks{manuel.ladra@usc.gal}}
\author[1]{Bernardo Leite da Cunha \thanks{bernardo.mariz@rai.usc.gal}}
\author[2]{Samuel A. Lopes \thanks{slopes@fc.up.pt}}
\affil[1]{Departamento de Matemáticas \& CITMAga, Universidade de Santiago de Compostela,
	15782, Spain.} 
\affil[2]{CMUP, Departamento de Matem\'atica, Faculdade de Ci\^encias, Universidade do Porto, Rua do Campo Alegre s/n, 4169--007 Porto, Portugal.}

\date{}
\begin{document}
	
	\maketitle
	
	\abstract{	We study compatible Lie algebras from algebraic and representation-theoretic points of view, obtaining counterexamples to some fundamental theorems from classical Lie algebra theory, namely the theorems of Lie, Weyl and Levi.
		
		We also classify the two-dimensional compatible Lie algebras up to isomorphism and explore their representation theory, presenting families of indecomposable non-semisimple representations, showing that the solvable two-dimensional compatible Lie algebras have wild representation type, and classifying all irreducible finite-dimensional line representations. Finally, we prove a Clebsch--Gordan decomposition for tensor products of finite-dimensional irreducible line representations.
}

\tableofcontents
\enlargethispage{2\baselineskip}
	
	\section{Introduction}

	Given a vector space $\g$ endowed with two Lie products $\sqb{-,-}$ and $\cb{-,-}$, we say that they are \emph{compatible} if each linear combination $\lambda \sqb{-,-}+\mu \cb{-,-}$ is still a Lie product, or equivalently, if the following identity  holds for any $x, y, z \in \g$
	\begin{align}\label{mixedjacobi:0}
		\begin{split}
			& \cb{\sqb{x,y},z}+\cb{\sqb{y,z},x}+\cb{\sqb{z,x},y} \\
			& + \sqb{\cb{x,y},z}+\sqb{\cb{y,z},x}+\sqb{\cb{z,x},y}=0.
		\end{split}
	\end{align}
	A vector space with two compatible Lie products is called a \emph{compatible Lie algebra}.	These structures arise in several areas of mathematics and mathematical physics, including the study of the classical Yang--Baxter equation \cite{CompatibleYangBaxter}, integrable equations of the principal chiral model type \cite{Golubchik2002}, elliptic theta functions \cite{CompatibleEliptic}, and loop algebras over Lie algebras \cite{Golubchik2005}.
	
	When setting $\lambda=1$ instead of considering general linear combinations $\lambda \sqb{-,-}+\mu \cb{-,-}$, the resulting product can be seen as an infinitesimal deformation of $\sqb{-,-}$. This follows by observing that \Cref{mixedjacobi:0} is simply the cocycle identity for the adjoint module.  Thus, the study of compatible Lie algebras is also connected with deformation theory.
	
	A related notion is that of bi-Hamiltonian structures, that is, pairs of compatible Poisson brackets defined on the same manifold. These play an essential role in the theory of integrable systems in mathematical physics. Operads of compatible Lie  algebras and bi-Hamiltonian algebras have been studied by Dotsenko and Khoroshkin in \cite{CharacterFormulasBiHamiltonian06}.


	A broader notion is that of compatible algebraic structures. Two algebraic structures of the same type (i.e.\ both associative algebras, Lie algebras, Leibniz algebras, etc.) $(V, \circ)$ and $(V, \ast)$ with the same underlying vector space are said to be compatible if $(V, \lambda \cdot \circ + \mu  \cdot  \ast)$ has the same type of algebraic structure as $(V, \circ)$ and $(V, \ast)$ for any scalars $\lambda, \mu$. Some non-Lie examples in the literature include compatible associative algebras in \cite{OdesskiiSokolovPairsCompatibleAssoc2006} and \cite{OdesskiiSokolovPairsCompatibleAssoc2008}, compatible associative bialgebras in \cite{CompatibleAlgebrasMarquez}, compatible Lie bialgebras in \cite{Wu2015} and compatible Leibniz algebras in \cite{CompLeibniz2025} and \cite{CohomDeformCompLeibninz2023}. General compatible structures have been studied from an operadic point of view in \cite{Strohmayer2008}.
	
	The classification of compatible Lie algebras is an increasingly active area of research. The algebraic classification of low-dimensional nilpotent compatible Lie algebras was obtained in \cite{NCLclassification25}, and the geometric classification in \cite{Abdurasulov2025}.
	A method of constructing solvable compatible Lie algebras from nilpotent ones was presented by Ouaridi, Navarro, Omirov, and Solijanova in \cite{Ouaridi2026}.

	This article is structured as follows. In \Cref{prelims_sect} we recall the definitions of compatible Lie algebras and their representations, and establish basic structural properties.
	In \Cref{counterexamples_sect}, we show that three classical results of Lie theory, namely Lie's Theorem, Weyl's complete reducibility Theorem and the Levi decomposition, fail in the compatible setting, highlighting the interest and more subtle nature of compatible Lie algebra representations.
	
	In \Cref{2d_sect}, we classify the two-dimensional compatible Lie algebras and explore their representation theory. For the solvable ones, we show that the representation type is wild. Focusing next on the unique simple compatible Lie algebra of dimension two, we obtain a family of finite-dimensional irreducible representations, which we name finite-dimensional irreducible line representations.
	
	We end with \Cref{tensor_sect}, wherein we study tensor products of finite-dimensional irreducible line representations, establishing a Clebsch–Gordan decomposition formula.
	
	\section{Preliminaries}\label{prelims_sect}
	
	\subsection{Compatible Lie algebras}
	
	In this section, we start by defining \emph{compatible Lie algebras} and recalling their basic properties. We let $\KK$ be an arbitrary field of characteristic different from $2$.
	
	
	\begin{prop}\label{compatibleequiv}
		Let $\underline{\g}=(\g,\sqb{-,-})$ and $\undertilde{\g}=(\g,\cb{-,-})$ be two Lie algebras over the same vector space $\g$. Then the following conditions are equivalent:
		\begin{enumerate}[(i)]
			\item\label{comp1} $(\g, \dbl{-,-})$ is a Lie algebra, where $\dbl{x,y}=\sqb{x,y}+\cb{x,y}$ for all $x,y \in \g$;
			\item\label{comp2} $(\g, \dbl{-,-}_{\lambda,\lambda'})$ is a Lie algebra for all $\lambda, \lambda' \in \KK$, where $\dbl{x,y}_{\lambda,\lambda'}=\lambda\sqb{x,y}+\lambda'\cb{x,y}$ for all $x,y \in \g$;
			\item\label{comp3} The following identity (named the \emph{mixed Jacobi identity}) holds for all $x,y,z \in \g$:
			\begin{align}\label{mixedjacobi}
				\begin{split}
					& \cb{\sqb{x,y},z}+\cb{\sqb{y,z},x}+\cb{\sqb{z,x},y} \\
					& + \sqb{\cb{x,y},z}+\sqb{\cb{y,z},x}+\sqb{\cb{z,x},y}=0.
				\end{split}
			\end{align}
		\end{enumerate}
	\end{prop}

With this, we have the definition of compatible Lie algebras.
	
	\begin{dfn}
		A \emph{compatible Lie algebra} is a triple $(\g,\sqb{-,-},\cb{-,-})$, where $\underline{\g}=(\g,\sqb{-,-})$ and $\undertilde{\g}=(\g,\cb{-,-})$ are Lie algebras satisfying any of the three equivalent conditions in \Cref{compatibleequiv}.
		
	\end{dfn}
	
	
	We refer the reader to \cite[Section 2]{NCLclassification25} for some basic concrete examples of compatible Lie algebras.


	\begin{dfn}
		A \emph{compatible Lie algebra homomorphism} between two compatible Lie algebras $(\g,\sqb{-,-}_\g,\cb{-,-}_\g)$ and $(\hh,\sqb{-,-}_\hh,\cb{-,-}_\hh)$ is a linear map $\varphi\colon\g \rightarrow \hh$ which is a Lie algebra homomorphism between $\underline{\g}$ and  $\underline{\hh}$, and  a Lie algebra homomorphism between $\undertilde{\g}$ and  $\undertilde{\hh}$.		
			As usual, an invertible homomorphism is an \emph{isomorphism}.
	\end{dfn}

	The classical definitions carry over to the compatible case:
	
	\begin{dfn}\hfill
		\begin{itemize}
			\item A \emph{subalgebra} of a compatible Lie algebra $\g$ is a vector subspace of $\g$ which is closed for both products.
			\item An \emph{ideal} $\mathfrak{i}$ of a compatible Lie algebra $\g$ is a vector subspace such that
			$
			\sqb{\mathfrak{i},\g}, \cb{\mathfrak{i},\g} \subseteq \mathfrak{i}. 
			$
			\item A compatible Lie algebra is said to be \emph{abelian} if both its products are trivial, or equivalently, if both of its component Lie algebras are abelian Lie algebras.
			\item A \emph{simple} compatible Lie algebra is one that is non-abelian and contains no non-trivial ideals.
			\item 	The \emph{centre} of a compatible Lie algebra $\g$, denoted by $Z(\g)$, is the ideal defined by
			\[Z(\g)=\cb{x \in \g \mid \sqb{x,\g}=0=\cb{x,\g} } = Z(\underline{\g})\cap Z(\undertilde{\g}) .\]
		\end{itemize}		
	\end{dfn}
	
	As expected, the kernel of a homomorphism is an ideal of the domain and the image of a homomorphism is a subalgebra of the codomain. The notion of quotient is well defined and the usual isomorphism theorems hold.


%
%
%
%
%

\subsection{Representations of compatible Lie algebras}

As in the classical case, there is a notion of representation for compatible Lie algebras. In this section, we give the basic definitions in the compatible context.

\begin{dfn}[\cite{Wu2015}]\label{representation}
	A \emph{representation} of a compatible Lie algebra $\g$ is a triple $(V,\rho,\mu)$, where $(V,\rho)$ is a representation of $\underline{\g}$, $(V,\mu)$ is a representation of $\undertilde\g$, and $(V,\rho+\mu)$ is a representation of~$(\g,\dbl{-,-})$.
\end{dfn}

\begin{rmk}
	The previous definition is equivalent to requiring that the following identities hold for all $x, y \in \g$.
	\begin{align}
		\rho(\sqb{x,y})&=\rho(x)\rho(y)-\rho(y)\rho(x); \label{rep:id1} \\
		\mu(\cb{x,y})&=\mu(x)\mu(y)-\mu(y)\mu(x);  \label{rep:id2} \\
		\rho(\cb{x,y})+\mu(\sqb{x,y})&= \rho(x)\mu(y)-\mu(y)\rho(x)+\mu(x)\rho(y)-\rho(y)\mu(x). \label{rep:id3}
	\end{align}
\end{rmk}

\begin{example}[Adjoint representation]
	Let $\underline\ad$ denote the adjoint representation of $(\g,\sqb{-,-})$ and $\undertilde\ad$ denote the adjoint representation of $(\g,\cb{-,-})$. Then the triple $(\g,\underline\ad,\undertilde\ad)$ is a representation of $\g$ called the \emph{adjoint representation}. The compatibility condition between $\underline\ad$ and $\undertilde\ad$ is precisely the mixed Jacobi identity.
\end{example}

The definitions of \emph{subrepresentation}, \emph{irreducible representation} and \emph{indecomposable representation} are the obvious generalisations of the classical counterparts.

%
%
%
%

\subsection{Solvable and semisimple compatible Lie algebras}

Note that the results of this section hold, 
 mutatis mutandis, for compatible algebras of any type for which a notion of commutator of subalgebras is available. In fact, the proofs are very similar to the classical case, so we do not include them in the text.

Recall the commutator of subalgebras 
\[
\dbl{\mathfrak{s},\mathfrak{t}} = \ab{\sqb{s,t},\cb{s,t} \mid s \in \mathfrak{s}, t \in \mathfrak{t}}.
\]

Using this commutator, we may define the following series
\[
\g^{(0)} \supseteq \g^{(1)} \supseteq \cdots \supseteq \g^{(i)} \supseteq \cdots ,
\]
where 
\[
\g^{(0)} := \g \text{ and } \g^{(i+1)} = \dbl{\g^{(i)},\g^{(i)}}.
\]
Each term of this series is an ideal of the previous one, 
and each quotient is abelian. A straightforward induction proof shows that 
\[
(\g^{(i)})^{(j)} = \g^{(i+j)}.
\]

\begin{dfn}
	A compatible Lie algebra is said to be \emph{solvable} if $\g^{(i)}=0$ for some $i \in \NN$.
\end{dfn}

\begin{lemma}\label{derived_properties}
	Let $\g$ be a compatible Lie algebra. Then
	\begin{enumerate}[(a)]
		\item if $\varphi\colon \g \rightarrow \hh$ is a surjective homomorphism, then $\varphi(\g^{(i)})=\hh^{(i)}$ for each $i \in \NN$;
		\item if $\hh$ is a subalgebra of $\g$, then $\hh^{(i)}\subseteq \g^{(i)}$ for each $i \in \NN$;
		\item if $\mathfrak{i}$ is an ideal of $\g$, then $\g$ is solvable if and only if both $\mathfrak{i}$ and $\g/\mathfrak{i}$ are solvable.
	\end{enumerate}
\end{lemma}

{

With the aid of this lemma, we are now able to give the following definition and to prove the next corollary.

\begin{dfn}
	Let $\g$ be a finite-dimensional compatible Lie algebra. The set of all solvable ideals of $\g$ is nonempty and has a unique maximal element called its \emph{radical}, which we will denote by $\rad(\g)$. We say that a compatible Lie algebra $\g$ is \emph{semisimple} if $\rad(\g)=\cb{0}$.
\end{dfn}




\begin{cor}\hfill
	\begin{enumerate}[(a)]
		\item For any compatible Lie algebra $\g$, the compatible Lie algebra $\g/\rad(\g)$ is semisimple.
		\item A direct sum of simple compatible Lie algebras is semisimple.
	\end{enumerate}
\end{cor}
}

\section{Counterexamples to the theorems of Lie, Weyl and Levi
}\label{counterexamples_sect}

{Three} of the most well-known theorems from the theory of Lie algebras are {Lie's Theorem,} Weyl's Theorem and Levi's Theorem. 


{
\begin{theorem}[Lie's Theorem]
	Suppose $\KK$ is algebraically closed and of characteristic $0$. The finite-dimensional irreducible representations of a finite-dimensional solvable Lie algebra have dimension $1$.
\end{theorem}
}

\begin{theorem}[Weyl's Theorem]
	Suppose $\KK$  has characteristic $0$. Any finite-dimensional representation of a finite-dimensional simple Lie algebra is completely reducible. 
\end{theorem}

\begin{theorem}[Levi's Theorem]
	Suppose $\KK$  has characteristic $0$. A finite-dimensional Lie algebra $\g$ has a subalgebra isomorphic to $\g/\rad(\g)$ so that $\g$ admits a semidirect product decomposition $\g \simeq \mathfrak{s} \ltimes \rad(\g)$ where $\mathfrak{s}$ is a semisimple subalgebra of $\g$. Such a subalgebra $\mathfrak{s}$ is called a \emph{Levi subalgebra of $\g$}. 
\end{theorem}


In the case of compatible Lie algebras, these theorems fail to hold. To construct counterexamples, we {make use of} the following compatible Lie algebra{s} (see \Cref{2d_comp_class}).

\begin{dfn}\label{CL2def}
	{Let $CL_{2,1}$ be the two-dimensional compatible Lie algebra with basis ${x, y}$ and nonzero products between basis elements
	\[
	\sqb{x,y}=x.
	\]
}
	
	Let $CL_{2,4}$ be the two-dimensional compatible Lie algebra with basis ${x, y}$ and {nonzero} products given by
	\[
	\sqb{x,y}=x, \quad \cb{x,y}=y.
	\]
\end{dfn}

{The compatible Lie algebras $CL_{2,1}$ and $CL_{2,4}$ are solvable and simple, respectively, and they will provide a counterexample to Lie's Theorem and to Weyl's and Levi's Theorems, respectively.}

\subsection{The failure  of Lie's and Weyl's theorems}

\begin{prop}\label{Liecounter}
	The compatible Lie algebra $CL_{2,1}$ admits irreducible representations of each finite dimension greater than $1$. Therefore, Lie's Theorem fails to hold in the class of compatible Lie algebras.
\end{prop}
\begin{proof}
	Some families of such representations will be constructed in \Cref{repsCLothers_sect}.
\end{proof}

We note that we show a stronger result, that all solvable two-dimensional compatible Lie algebras have wild representation type.

\begin{prop}\label{Weylcounter}
	The compatible Lie algebra $CL_{2,4}$ admits non-simple and indecomposable representations of each finite dimension greater than $1$. Therefore, Weyl's Theorem fails to hold in the class of compatible Lie algebras.
\end{prop}

\begin{proof}
	Some families of such representations will be constructed in \Cref{repsCL2_sect}.
\end{proof}

\subsection{Counterexample to Levi's theorem}

Let us note that to find a counterexample to Levi's Theorem it suffices to find a compatible Lie algebra $\mathfrak{g}$ that is isomorphic to $\mathfrak{g}/\mathrm{rad}(\mathfrak{g})\oplus\mathrm{rad}(\mathfrak{g})$ as a vector space but admits no subalgebra isomorphic to $\mathfrak{g}/\mathrm{rad}(\mathfrak{g})$.

The simplest possible case is a one-dimensional abelian extension, so let $V=\CC z$. A direct computation shows that the only possible one-dimensional representations of $CL_{2,4}$ are the following:
\[
\rho(x)z = 0, \quad \rho(y)z = -\lambda z, \quad \mu(x)z = \lambda z, \quad \mu(y)z=0.
\]
For simplicity, we take $\lambda = 1$.

To build an extension, we use a cocycle of $\g$ with values in $(V,\rho,\mu)$, i.e.~a pair of bilinear alternating maps $\omega=(\underline{\omega},\undertilde{\omega})\colon\g\times \g \rightarrow~V$ that satisfy the cocycle identities:
\begin{align*}
	0=& \: \underline{\omega}(\sqb{x,y},z)+ \underline{\omega}(\sqb{z,x},y) + \underline{\omega}(\sqb{y,z},x), \\
	0=& \: \undertilde{\omega}(\cb{x,y},z)+ \undertilde{\omega}(\cb{z,x},y) + \undertilde{\omega}(\cb{y,z},x),  \\
	0=& \: \underline{\omega}(\cb{x,y},z)+ \underline{\omega}(\cb{z,x},y) + \underline{\omega}(\cb{y,z},x)  +\: \undertilde{\omega}(\sqb{x,y},z)+ \undertilde{\omega}(\sqb{z,x},y) + \undertilde{\omega}(\sqb{y,z},x),
\end{align*}
for all $x, y, z \in \g$.

We refer the reader to Liu, Sheng, and Bai~\cite{Liu2023} for the general theory of compatible cohomology.

Since $CL_{2,4}$ is two-dimensional, any pair of bilinear alternating maps satisfies these identities, and since $V$ is one-dimensional, the only such maps are those of the form
\[
\underline{\omega}(x,y) = pz, \quad \undertilde{\omega}(x,y) = qz, \quad\text{for some } p,q \in \CC.
\]
We let $p=-1$, $q = 0$, so that $\underline{\omega}(x,y) = -z,~\undertilde{\omega}(x,y) = 0$. 


The general definition of abelian extension given in \cite[Section 4.2.]{Liu2023} appears to contain a sign error; with the corrected formula, the products on $\g=\ab{x,y,z}=CL_{2,4} \oplus \CC z$ are as follows:
\begin{align*}
	\sqb{(a,u),(b,v)}_\g &= (\sqb{a,b}_{CL_{2,4}}, \rho(a)v-\rho(b)u-\underline{\omega}(a,b)) \\
	\cb{(a,u),(b,v)}_\g &= (\cb{a,b}_{CL_{2,4}}, \mu(a)v-\mu(b)u-\undertilde{\omega}(a,b)), \text{ for } a,b \in CL_{2,4}, u,v \in \CC z.
\end{align*}
With our particular choices of representation and cocycle, we end up with the following non-zero products on the basis elements:
\begin{alignat*}{2}
	\sqb{x,y} &= x + z, \quad &  \sqb{y,z} &= -z, \\
	\cb{x,y} &= y,  & \cb{x,z}&=z.
\end{alignat*}

We now proceed to show that this algebra is indeed a counterexample to Levi's Theorem, by proving the following:

\begin{prop}\label{Levicounter}
	The compatible Lie algebra $\g$ with basis ${x,y,z}$ and relations
	\begin{align*}
		&\sqb{x,y} = x + z, \quad \sqb{y,z} = -z, \qquad \cb{x,y} = y, \quad \cb{x,z}=z,
	\end{align*}
	has no subalgebra isomorphic to $\g/\rad(\g)$. Thus, Levi's Theorem fails to hold in the class of compatible Lie algebras.
\end{prop} 

\begin{proof}
	It can be directly checked that $\ab{z}$ is a solvable ideal of $\g$. 
	Since $\g / \ab{z} \cong CL_{2,4}$ is semisimple, it follows that $\rad(\g)=\ab{z}$.
	
	Now, assume that we have a subalgebra $\g_0=\ab{X=ax+by+cz, Y=a'x+b'y+c'z}$ of $\g$ such that $\g \cong \ab{z} \oplus \g_0$ and $\g_0 \cong CL_{2,4}$. Without loss of generality, we may assume that $\sqb{X,Y}=X$ and $\cb{X,Y}=Y$.
	
	Computing the expression $\sqb{X,Y}$ we obtain
	\begin{align*}
		\sqb{X,Y} &= \sqb{ax+by+cz,a'x+b'y+c'z} = ab'\sqb{x,y}+ba'\sqb{y,x}+bc'\sqb{y,z}+cb'\sqb{z,y}\\
		&= (ab'-ba')(x+z)+(cb'-c'b)z.
	\end{align*}
	Setting this equal to $X=ax+by+cz$ and comparing coefficients yields the following linear system of equations
	\begin{align*}
		\begin{cases*}
			a = ab'-ba' \\
			b = 0 \\
			c=ab'-ba'+cb'-bc'
		\end{cases*} \Leftrightarrow 
		\begin{cases*}
			a = ab' \\
			b = 0 \\
			c=ab'+cb'
		\end{cases*},
	\end{align*}
	which yields either $b'=1$ or $a=c=0$. Since the set $\cb{z,X,Y}$ is required to be linearly independent, we must have $b'=1$ and thus $c=a+c$, so $a=0$ and $X=cz$, a contradiction.
\end{proof}

\section{Two-dimensional compatible Lie algebras and their representations}\label{2d_sect}

In this section, we classify all two-dimensional compatible Lie algebras and discuss their representations. In the case of two-dimensional solvable compatible Lie algebras, we show that they have wild representation type. We next focus on the representations of the only simple algebra of dimension $2$, obtaining a family of counterexamples to Weyl's Theorem and classifying a family of irreducible representations that resemble the finite-dimensional representations of $\mathfrak{sl}_2$. 

\subsection{Classification of two-dimensional compatible Lie algebras}\label{2d_comp_class_subsect}

We begin by providing a complete classification of two-dimensional compatible Lie algebras over an arbitrary field $\KK$ of characteristic not $2$. 

\begin{theorem}\label{2d_comp_class}
	In dimension $2$, there are the following isomorphism classes of compatible Lie algebras:
	\begin{itemize}
		\item $CL_{2,0}$, abelian;
		\item $CL_{2,1}$, with nonzero brackets $\sqb{e_1,e_2}=e_1$;
		\item $CL_{2,2}$, with nonzero brackets $\cb{e_1,e_2}=e_1$;
		\item $CL_{2,3}^\alpha$, $\alpha \in \KK^\times$, with nonzero brackets $\sqb{e_1,e_2}=e_1$, $\cb{e_1,e_2}=\alpha e_1$;
		\item $CL_{2,4}$, with nonzero brackets $\sqb{e_1,e_2}=e_1$, $\cb{e_1,e_2}= e_2$.
	\end{itemize}
\end{theorem}

\begin{rmk}
	Note that the mixed Jacobi identity (\Cref{mixedjacobi}) holds trivially for any pair of Lie brackets in a two-dimensional vector space. Thus, the algebras listed in \Cref{2d_comp_class} are indeed compatible Lie algebras.
\end{rmk}

\begin{proof}[Proof of \Cref{2d_comp_class}]
	Let $\g$ be a compatible Lie algebra of dimension $2$.
	
	Suppose first that $\underline\g$ is not abelian. Then $\sqb{\g, \g}$ has dimension $1$, say $\sqb{\g, \g} = \KK e_1$ for some nonzero element $e_1 \in \g$. Choose any $e_2 \in \g$ such that $e_1, e_2$ is a basis of $\g$. In particular, since $\sqb{e_1,e_2}=\lambda e_1$ for some $\lambda \in \KK^\times$, we can replace $e_2$ with $\lambda^{-1}e_2$ and thus assume that $\sqb{e_1,e_2}= e_1$.
	
	\underline{Case 1:} $\cb{\g, \g} \subseteq \KK e_1$.
	
	Then $\cb{e_1,e_2} = \alpha e_1$ for some $\alpha \in \KK$. If $\alpha = 0$ then $\g \cong CL_{2,1}$, and if $\alpha \neq 0$ then $\g \cong CL_{2,3}^\alpha$.
	
	\underline{Case 2:} $\cb{\g, \g} \not\subseteq \KK e_1$.
	
	Then, as $\dim \cb{\g, \g} = 1$, it follows that $\g = \KK e_1 \oplus \KK e_2$ where $\cb{\g, \g} = \KK e_2$. So $\cb{e_1,e_2} = \lambda e_2$, for some $\lambda \in \KK^\times$. If we replace $e_1$ with $\lambda^{-1}e_1$  we get $
	\sqb{e_1,e_2} = e_1 \text{ and } \cb{e_1,e_2} = e_2,
	$
	so $\g \cong CL_{2,4}$.
	
	Now suppose that $\underline\g$ is abelian. If $\undertilde\g$ is not abelian, then by the argument at the beginning of the proof, we have $\g \cong CL_{2,2}$. If both products are abelian, $\g \cong CL_{2,0}$.
	
	In terms of isomorphism classes, suppose that $\varphi\colon\g \rightarrow \hh$ is an isomorphism of compatible Lie algebras. Then $\underline\g$ is abelian if and only if $\underline\hh$ is abelian, and similarly for $\undertilde\g$ and $\undertilde\hh$. Thus, the algebras $CL_{2,0}$, $CL_{2,1}$ and $CL_{2,2}$ are in distinct isomorphism classes; moreover, $CL_{2,3}^\alpha$ and $CL_{2,4}$ are not isomorphic to any of the former three algebras. Hence, it remains to show that $CL_{2,3}^\alpha\cong CL_{2,3}^\beta$ if and only if $\alpha = \beta$, and that $CL_{2,4} \not\cong CL_{2,3}^\alpha$ for all $\alpha \in \KK^\times$.
	
	Suppose that  $\varphi\colon\g \rightarrow \hh$ is an isomorphism of compatible Lie algebras, where $\g = CL_{2,3}^\alpha$ and $\hh = CL_{2,3}^\beta$ or $\hh = CL_{2,4}$.
	Then, 
	\[
	0 \neq \varphi(e_1) = \varphi(\sqb{e_1,e_2}_\g) = \sqb{\varphi(e_1),\varphi(e_2)}_\hh \in \sqb{\hh, \hh} = \KK e_1,
	\] 
	so $\varphi(e_1) = \lambda e_1$ for some $\lambda \in \KK^\times$.
	Similarly, $\lambda e_1 = \varphi (e_1) \in \cb{\hh,\hh}$. It follows that  $\hh$ cannot be $CL_{2,4}$. 
	
	So assume that $\hh = CL_{2,3}^\beta$. Write $\varphi(e_2) = \gamma e_1 + \delta e_2$, with $\gamma,\delta \in \KK$. We have
	\begin{align*}
		\lambda e_1 &= \varphi(e_1) = \varphi(\sqb{e_1,e_2}_\g) =  \sqb{\varphi(e_1),\varphi(e_2)}_\hh = \lambda\sqb{e_1,\gamma e_1 + \delta e_2} = \lambda\delta e_1,
	\end{align*}
	therefore, $\delta = 1$. Finally,
	\begin{align*}
		\alpha\lambda e_1 &= \varphi(\alpha e_1) = \varphi(\cb{e_1,e_2}_\g) =  \cb{\varphi(e_1),\varphi(e_2)}_\hh = \lambda\cb{e_1,\gamma e_1 + e_2} = \lambda\beta e_1,
	\end{align*}
	so $\alpha = \beta$.
\end{proof}

Having obtained the classification, we add a couple of remarks.

\begin{rmk}\hfill\label{classification_remarks}
	\begin{itemize}
		\item (See \cite[Definition 2.5]{NCLclassification25}) $CL_{2,1}$ is skew-isomorphic to $CL_{2,2}$ and $CL_{2,3}^\alpha$ is skew-isomorphic to $CL_{2,3}^{1/\alpha}$;
		\item $CL_{2,0}$ is the only two-dimensional nilpotent compatible Lie algebra;
		\item $CL_{2,1}$, $CL_{2,2}$ and $CL_{2,3}^\alpha$ are all solvable and non-nilpotent. In each case, the first term of the derived series is $\ab{e_1}$, and the derived series terminates at the following term;
		\item $CL_{2,4}$ is the only two-dimensional simple compatible Lie algebra.
	\end{itemize}
\end{rmk}

\subsection{Representations of two-dimensional compatible Lie algebras}\label{2d_reps_subsect}

We start this section by stating and proving a lemma that will be useful. 

\begin{lemma}\label{combination_reps}
	Given a compatible Lie algebra $(\g, \sqb{-,-},\cb{-,-})$, then, for any $\lambda_1, \lambda_2 \in \KK$, the triple $(\g', \sqb{-,-}', \cb{-,-}')$ is a compatible Lie algebra, where $\g' = \g$ as vector spaces, and the products of $\g'$ are given by $\sqb{-,-}'=\sqb{-,-}$ and $\cb{-,-}' = \lambda_1\sqb{-,-}+\lambda_2\cb{-,-}$. 
	
	Moreover, if $(V,\rho,\mu)$ is a compatible representation of $\g$, then $(V,\rho', \mu')$ is a representation of $\g'$, where $\rho'=\rho$ and $\mu' = \lambda_1\rho+\lambda_2\mu$.	Furthermore, the subrepresentations of $(V,\rho, \mu)$ are also subrepresentations of $(V,\rho',\mu')$, and in case $\lambda_2 \neq 0$, the converse is also true.
\end{lemma}

\begin{proof}
	The first part of the statement is trivial, as any linear combination of $\sqb{-,-}'$ and $\cb{-,-}'$ is also a linear combination of $\sqb{-,-}$ and $\cb{-,-}$, and thus a Lie product.
	
	Now, to prove the second part of the statement, we just verify that $\rho'$ and $\mu'$ satisfy 
	\Crefrange{rep:id1}{rep:id3}.
	
	Since $\sqb{-,-}'=\sqb{-,-}$ and $\rho'=\rho$, it is immediate that \Cref{rep:id1} holds.
	
	Now, let $x,y \in \g' (= \g)$. Consider the left-hand side of \Cref{rep:id2} for $\mu'$ and $\cb{-,-}'$:
	\begin{align*}
		\mu'(\cb{x,y}') &=  \mu'(\lambda_1\sqb{x,y}+\lambda_2\cb{x,y}) \\
		&= \lambda_1\rho(\lambda_1\sqb{x,y}+\lambda_2\cb{x,y})+\lambda_2\mu(\lambda_1\sqb{x,y}+\lambda_2\cb{x,y}) \\
		&= \lambda_1^2\rho(\sqb{x,y})\lambda_1\lambda_2\big(\rho(\cb{x,y})+\mu(\sqb{x,y})\big) +  \lambda_2^2\mu(\cb{x,y}).
	\end{align*}
	Focus now on the right-hand side:
	\begin{align*}
		\mu'(x)\mu'(y)-\mu'(y)\mu'(x) &= \nb{ \lambda_1\rho(x)+\lambda_2\mu(x)}\nb{ \lambda_1\rho(y)+\lambda_2\mu(y)} \\ &\qquad -  \nb{ \lambda_1\rho(y)+\lambda_2\mu(y)}\nb{ \lambda_1\rho(x)+\lambda_2\mu(x)} \\
		&= \lambda_1^2\rho(x)\rho(y) + \lambda_1\lambda_2\rho(x)\mu(y) + \lambda_2\lambda_1\mu(x)\rho(y) + \lambda_2^2\mu(x)\mu(y) \\
		&\qquad -\lambda_1^2\rho(y)\rho(x) - \lambda_1\lambda_2\rho(y)\mu(x) - \lambda_2\lambda_1\mu(y)\rho(x) - \lambda_2^2\mu(y)\mu(x) \\
		&=\lambda_1^2\rho(\sqb{x,y})+\lambda_1\lambda_2\big(\rho(\cb{x,y})+\mu(\sqb{x,y})\big) +  \lambda_2^2\mu(\cb{x,y}),
	\end{align*}
	where in the last equality, we group like-terms together and use \Cref{rep:id1,rep:id2,rep:id3} with $\rho$ and $\mu$. Since both sides are equal, the identity is satisfied.
	
	We proceed in the same way for the third identity, and although the computations are longer, the result still holds.
	
	Finally, the last part of the statement can be deduced by the fact that if $W$ is a subrepresentation of $(V,\rho,\mu)$, meaning it is invariant under $\rho(x)$ and $\mu(x)$ for all $x \in \g$, then it must also be invariant under any linear combination of those maps. In particular, it is invariant under $\mu' = \lambda_1\rho+\lambda_2\mu$, and thus it is a subrepresentation of $(V,\rho',\mu')$. Conversely, if $\lambda_2 \neq 0$, then $(\g, \sqb{-,-},\cb{-,-})$ and $(V,\rho,\mu)$ are obtained from $(\g', \sqb{-,-}',\cb{-,-}')$ and $(V',\rho',\mu')$ by the process above, as $\sqb{-,-} = \sqb{-,-}'$ and $\cb{-,-} = -\lambda_1 \lambda_2^{-1}\sqb{-,-}' + \lambda_2^{-1}\cb{-,-}'$.
\end{proof}

{
\begin{rmk}
	One may associate a diagram to a representation $(V,\rho,\mu)$ of a compatible Lie algebra $\g$ in the following way: Take a set of vertices in bijection with a basis ${v_0,\ldots,v_n}$ of $V$, and for each $v_i$ draw an arrow from $v_i$ to $v_j$ with label $\lambda_{ij}$ if $\lambda_{ij} \neq 0$ and $\lambda_{ij}$ is the coefficient of $v_j$ in the expression for $\rho(z)v_i$ or $\mu(z)v_i$ for $z$ in the basis of $\g$. We may also
	 assign colours to the arrows based on the specific map-basis element pair that generates them. 
	 In what follows, we will work with a compatible Lie algebra of dimension two, either $CL_{2,1}$ or $CL_{2,4}$. To avoid overloading notation, we will rename $e_1$ and $e_2$ to $x$ and $y$, respectively,  and we will use the following colours:
	\[
	\color{blue} \rho(x) = \text{blue} \quad \color{cyan} \rho(y)=\text{cyan} \quad \color{red} \mu(x)=\text{red} \quad \color{orange} \mu(y)=\text{orange}
	\]
	The diagrams are especially clear and useful whenever we are working with representations where basis elements are mapped to other basis elements up to a scalar, e.g.~$\rho(x)v_i = \lambda_{ij}v_j$.
\\	
	We will provide many examples of such diagrams in what follows. 
\end{rmk}}

{
We omit the computations to verify that each is actually a representation of the corresponding algebra, as they are simple yet cumbersome. One can verify this fact by computing \Cref{rep:id1,rep:id2,rep:id3}  on an arbitrary basis vector.

{
	\subsection{Representations of the solvable compatible Lie algebras of dimension \texorpdfstring{$2$}{2}}\label{repsCLothers_sect}
}

\Cref{combination_reps} allows us to reduce the study of the representations of $CL_{2,1}$, $CL_{2,2}$ and $CL_{2,3}^\alpha$ to the study of a single one of these algebras.

\begin{cor}\label{transportation_reps}
	There is a bijective correspondence between the representations of $CL_{2,1}$, $CL_{2,2}$ and of $CL_{2,3}^\alpha$, with $\alpha \in \KK^\times$. Moreover, this bijection preserves subrepresentations.
\end{cor}

\begin{proof}
	Since the algebras $CL_{2,1}$ and $CL_{2,2}$ are skew-isomorphic (see \Cref{classification_remarks}), it is clear that if $(V,\rho,\mu)$ is a representation of one of them, then $(V,\mu,\rho)$ is a representation of the other. It is also clear that this correspondence preserves subrepresentations.
	
	A bijection from the representations of $CL_{2,3}^\alpha$ to the representations of $CL_{2,1}$ is given by an application of \Cref{combination_reps}, noting that $\cb{-,-}_{CL_{2,3}^\alpha} = \alpha \sqb{-,-}_{CL_{2,1}} + 1\cdot\cb{-,-}_{CL_{2,1}}$.
\end{proof}

{With this in mind, we choose to explore representations of $CL_{2,0}$ and $CL_{2,1}$.
}

{\subsubsection{Counterexamples to Lie's Theorem}}

{In this section, we prove \Cref{Liecounter} by constructing a family of irreducible representations of $CL_{2,1}$ for each finite dimension $n+1, n \geq 1$. We note that, although Lie's Theorem requires an algebraically closed field of characteristic $0$, our construction of irreducible representations works over an arbitrary field.	
}

{
For the remainder of this section, we always use $V$ to denote a vector space with basis $v_0, v_1, \ldots, v_n$.
}

{Throughout this section, we denote the generators of $CL_{2,1}$ by $x,y$ and the only nonzero product on the basis elements is $\sqb{x,y}=x$. 
}

{\begin{lemma}\label{CL21reps_zeros}
	Let $\rho(x)=\mu(x)=0 \in \End_\KK(V)$ and $\rho(y),\mu(y)\in \End_\KK(V)$ be arbitrary. Then, the triple $(V,\rho,\mu)$ is always a representation of $CL_{2,0}$ and of $CL_{2,1}$.
\end{lemma}
\begin{proof}
	We directly verify that, for $CL_{2,0}$ and $CL_{2,1}$, both sides of \Cref{rep:id1,rep:id2,rep:id3} are identically null whenever $\rho(x)$ and $\mu(x)$ are the zero endomorphism.
\end{proof}
}
{
This gives us the freedom to choose $\rho(y)$ and $\mu(y)$ so that the resulting representation is irreducible for arbitrary $n\geq 1$.}


\begin{example}
	Let $\rho(y)v_i = v_{i-1}$ and $\mu(y)v_i = v_{i+1}$, with the convention that $v_{-1}:=0 =:v_{n+1}$. This is a representation of $CL_{2,1}$ by the previous lemma, and we claim it is irreducible. In fact, we can apply $\rho(y)$ successively to any nonzero linear combination of basis vectors in order to obtain a multiple of $v_0$, and from there, applying $\mu(y)^i$ yields $v_i$, thereby generating all of $V$, and therefore the representation is irreducible.
	
	The diagram for this representation is as follows.
	
	
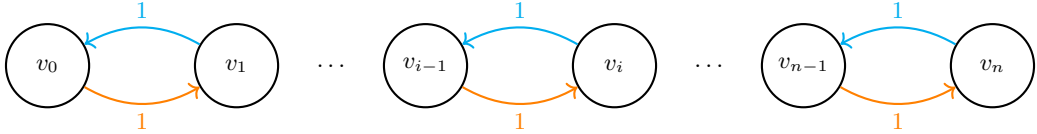
\begin{center}
		\begin{tikzpicture}[node distance={25mm},thick,main/.style = {draw, circle,minimum size=1.1cm}] 
			\node[main] (0) {$v_0$}; 
			\node[main] (1) [right of=0] {$v_1$};
			\node[main] (2) [right of=1] {$v_{i-1}$};
			\node[main] (3) [right of=2] {$v_{i}$};
			\node[main] (4) [right of=3] {$v_{n-1}$};
			\node[main] (5) [right of=4] {$v_n$};

			\path (1) -- node{$\ldots$} (2);
			\path (3) -- node{$\ldots$} (4);
			
			\draw[->,color=cyan] (1)  edge[bend right] node[midway, above]{$1$}  (0) ;
			\draw[->,color=orange] (0)  edge[bend right] node[midway, below]{$1$}  (1) ;

			\draw[->,color=cyan] (3)  edge[bend right] node[midway, above]{$1$}  (2) ;
			\draw[->,color=orange] (2)  edge[bend right] node[midway, below]{$1$}  (3) ;

			\draw[->,color=cyan] (5)  edge[bend right] node[midway, above]{$1$}  (4) ;
			\draw[->,color=orange] (4)  edge[bend right] node[midway, below]{$1$}  (5) ;

		\end{tikzpicture}
	\end{center}
\end{example}

{
\subsubsection{Wild representation type}
}

Informally, an algebra has wild representation type if the category of its modules ``contains'' the category of representations of the free associative algebra in two generators. We refer the reader to \cite{Makedonskii2013} for exact definitions and a discussion of wild representation type in the Lie algebra context.

There are numerous possibilities for representations constructed according to \Cref{CL21reps_zeros}, including indecomposable representations that are  not simple.
\begin{example}\label{indecomp_nonsimple_CL21}
	Let $\rho(y)v_i = v_{i-1}$, as before, but now set $\mu(y)=0$. This representation is indecomposable but non-simple. Its subrepresentations are exactly $V_k=\ab{v_i \mid i \leq k}$ with $0~\leq~k~\leq~n$, which form a chain of inclusions and therefore cannot be part of a direct sum decomposition. To see that the $V_k$ are the only subrepresentations, we show by induction on $k$ that an arbitrary element $v=\sum_{i=0}^k a_iv_i$ generates $V_k$, as long as $a_k\neq 0$.
	
	The base case $k=0$ is trivial. Now suppose we have $v=\sum_{i=0}^k a_iv_i$ with $a_k \neq 0$ and $k \geq 1$. We have that $\rho(y)v = \sum_{i=1}^k a_iv_{i-1} =: \sum_{i=0}^{k-1}b_iv_i$, and since $a_k \neq 0$, the coefficient $b_{k-1} = a_k \neq 0$, allowing us to apply the induction step to conclude that $\rho(y)v$ generates $V_{k-1}$ as a representation. But then, $v_k = a_k^{-1}(v-~\sum_{i=0}^{k-1} a_iv_{i})$ is in the subrepresentation generated by $v$, so indeed this element generates $V_k$. 
	
    The diagram for this representation is as follows.
	
	
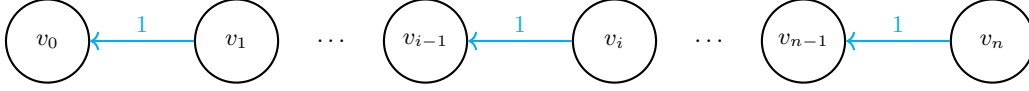
\begin{center}
		\begin{tikzpicture}[node distance={25mm},thick,main/.style = {draw, circle,minimum size=1.1cm}] 
			\node[main] (0) {$v_0$}; 
			\node[main] (1) [right of=0] {$v_1$};
			\node[main] (2) [right of=1] {$v_{i-1}$};
			\node[main] (3) [right of=2] {$v_{i}$};
			\node[main] (4) [right of=3] {$v_{n-1}$};
			\node[main] (5) [right of=4] {$v_n$};

			\path (1) -- node{$\ldots$} (2);
			\path (3) -- node{$\ldots$} (4);
			
			\draw[->,color=cyan] (1)  edge[] node[midway, above]{$1$}  (0) ;

			\draw[->,color=cyan] (3)  edge[] node[midway, above]{$1$}  (2) ;

			\draw[->,color=cyan] (5)  edge[] node[midway, above]{$1$}  (4) ;

		\end{tikzpicture}
	\end{center}
\end{example}

\begin{rmk}
	The previous example does not provide a counterexample to Weyl's Theorem, as the compatible Lie algebra $CL_{2,1}$ is not simple. Nonetheless, there are representations of the simple algebra $CL_{2,4}$ with the same properties, see \Cref{sectionWeylCounter}.
\end{rmk}

{

\begin{theorem}
	The solvable two-dimensional compatible Lie algebras ($CL_{2,0}$, $CL_{2,1}$, $CL_{2,2}$ and $CL_{2,3}^\alpha, \alpha \in \KK^\times$) have wild representation type.
\end{theorem}
\begin{proof}
	It suffices to show this for $CL_{2,0}$ and $CL_{2,1}$, since \Cref{transportation_reps} then allows us to transport this result to $CL_{2,2}$, and $CL_{2,3}^\alpha, \alpha \in \KK^\times$. Note that two representations $(V,\rho,\mu)$ and $(V',\rho',\mu')$ as in \Cref{CL21reps_zeros} are isomorphic if and only if there exists a vector space isomorphism $\varphi$ from $V$ to $V'$ such that $\varphi^{-1}\rho'(y)\varphi = \rho(y)$ and $\varphi^{-1}\mu'(y)\varphi = \mu(y)$.
	
	Thus, the isomorphism problem for the representations $(V,\rho,\mu)$ of this form is equivalent to the simultaneous similarity classification problem for pairs of square matrices of the same size, which is wild (see \cite{Friedland1983} for an algorithmic approach to this problem).
	
	Specifically, given the free associative $\KK$-algebra $\ab{X,Y}$ on two generators and a representation $(V,\tau)$ of $\ab{X,Y}$, we obtain a representation $(V,\rho,\mu)$ of $CL_{2,0}$ and of $CL_{2,1}$ of the type discussed in \Cref{CL21reps_zeros} where $\rho(x)=\mu(x)=0$, $\rho(y)=\tau(X)$, and $\mu(y)=\tau(Y)$.

	This correspondence preserves simplicity and indecomposability, and is faithful on isomorphism classes, thus proving the claim.	
\end{proof}

}


We now present a family of representations that is not as trivial, in the sense that $\rho(x)$, $\rho(y)$, $\mu(x)$ and $\mu(y)$ are all nonzero.

\begin{example}
	Let $\rho$ and $\mu$ be defined by
	\begin{align*}
		\rho(x)v_i = v_{i-1}, \quad \rho(y)v_i = (\beta+i)v_i, \quad \mu(x)v_i = v_{i-1}, \quad \mu(y)v_i = v_{i-1} + \gamma v_i,
	\end{align*}
again, with the convention that $v_{-1}=0$. The diagram for this representation is as follows.

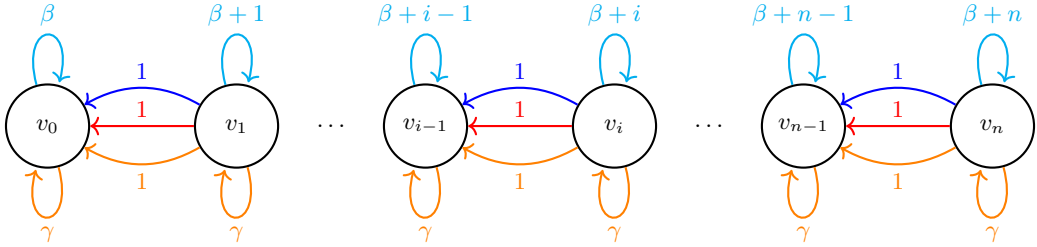
\begin{center}
	\begin{tikzpicture}[node distance={25mm},thick,main/.style = {draw, circle,minimum size=1.1cm}] 
		\node[main] (0) {$v_0$}; 
		\node[main] (1) [right of=0] {$v_1$};
		\node[main] (2) [right of=1] {$v_{i-1}$};
		\node[main] (3) [right of=2] {$v_{i}$};
		\node[main] (4) [right of=3] {$v_{n-1}$};
		\node[main] (5) [right of=4] {$v_n$};

		\path (1) -- node{$\ldots$} (2);
		\path (3) -- node{$\ldots$} (4);
		
		\draw[->,color=blue] (1)  edge[bend right] node[midway, above]{$1$}  (0) ;
		\draw[->,color=orange] (1)  edge[bend left] node[midway, below]{$1$}  (0) ;
		\draw[->,color=red] (1)  edge[] node[midway, above]{$1$}  (0) ;
		
		\draw[->,color=cyan] (0)  edge[loop above] node[midway, above]{$\beta$}  (0) ;
		\draw[->,color=orange] (0)  edge[loop below] node[midway, below]{$\gamma$}  (0) ;
		\draw[->,color=cyan] (1)  edge[loop above] node[midway, above]{$\beta+1$}  (1) ;
		\draw[->,color=orange] (1)  edge[loop below] node[midway, below]{$\gamma$}  (1) ;

		\draw[->,color=blue] (3)  edge[bend right] node[midway, above]{$1$}  (2) ;
		\draw[->,color=orange] (3)  edge[bend left] node[midway, below]{$1$}  (2) ;
		\draw[->,color=red] (3)  edge[] node[midway, above]{$1$}  (2) ;
		
		\draw[->,color=cyan] (2)  edge[loop above] node[midway, above]{$\beta+i-1$}  (2) ;
		\draw[->,color=orange] (2)  edge[loop below] node[midway, below]{$\gamma$}  (2) ;
		\draw[->,color=cyan] (3)  edge[loop above] node[midway, above]{$\beta+i$}  (3) ;
		\draw[->,color=orange] (3)  edge[loop below] node[midway, below]{$\gamma$}  (3) ;

		\draw[->,color=blue] (5)  edge[bend right] node[midway, above]{$1$}  (4) ;
		\draw[->,color=orange] (5)  edge[bend left] node[midway, below]{$1$}  (4) ;
		\draw[->,color=red] (5)  edge[] node[midway, above]{$1$}  (4) ;
		
		\draw[->,color=cyan] (4)  edge[loop above] node[midway, above]{$\beta+n-1$}  (4) ;
		\draw[->,color=orange] (4)  edge[loop below] node[midway, below]{$\gamma$}  (4) ;
		
		\draw[->,color=cyan] (5)  edge[loop above] node[midway, above]{$\beta+n$}  (5) ;
		\draw[->,color=orange] (5)  edge[loop below] node[midway, below]{$\gamma$}  (5) ;
		
	\end{tikzpicture}
\end{center}
We do not explicitly check that this is a representation. As in \Cref{indecomp_nonsimple_CL21}, this representation is indecomposable but non-simple, as its subrepresentations are exactly $V_k=\ab{v_i \mid i \leq k}$. By a straightforward eigenvalue argument, the isomorphism classes are parametrised by the pairs $(\beta,\gamma) \in \KK^2$.

\end{example}

\subsection{Representations of \texorpdfstring{$CL_{2,4}$}{CL24}}\label{repsCL2_sect}

In this section, we construct several representations of $CL_{2,4}$ satisfying different properties. A few of the families we construct provide a proof for~\Cref{Weylcounter}, see \Cref{sectionWeylCounter}.

To avoid excessive use of subscripts, the basis elements in \Cref{2d_comp_class} are renamed as $x$ and $y$, respectively, throughout this section.

\subsubsection{Families of counterexamples to Weyl's Theorem}\label{sectionWeylCounter}

In this section, we construct some families of finite-dimensional representations which are not irreducible, yet are indecomposable; hence, they are not semisimple.

\begin{example}\label{example_line1}
	
	Fixing a basis $\cb{v_0,\ldots, v_n}$ of a vector space $V$ of dimension $n+1$, the following is a representation of $CL_{2,4}$ (by convention, $v_{-1}=0$):
	\begin{align*}
		\rho(x) v_i = iv_{i-1}, &\quad  \rho(y)v_i=-(n-i)v_i, \quad \mu(x) v_i = (n-i)v_i, \quad  \mu(y)v_i=iv_{i-1}.
	\end{align*}
	
	The diagram for this representation is as follows:~
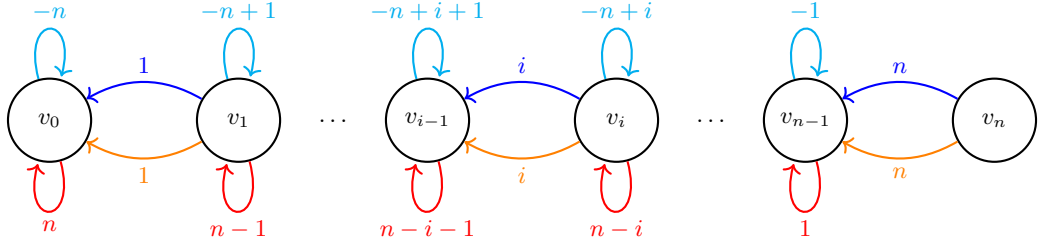
	\begin{center}
		\begin{tikzpicture}[node distance={25mm},thick,main/.style = {draw, circle,minimum size=1.1cm}] 
			\node[main] (0) {$v_0$}; 
			\node[main] (1) [right of=0] {$v_1$};
			\node[main] (2) [right of=1] {$v_{i-1}$};
			\node[main] (3) [right of=2] {$v_{i}$};
			\node[main] (4) [right of=3] {$v_{n-1}$};
			\node[main] (5) [right of=4] {$v_n$};

			\path (1) -- node{$\ldots$} (2);
			\path (3) -- node{$\ldots$} (4);
			
			\draw[->,color=blue] (1)  edge[bend right] node[midway, above]{$1$}  (0) ;
			\draw[->,color=orange] (1)  edge[bend left] node[midway, below]{$1$}  (0) ;
			
			\draw[->,color=cyan] (0)  edge[loop above] node[midway, above]{$-n$}  (0) ;
			\draw[->,color=red] (0)  edge[loop below] node[midway, below]{$n$}  (0) ;
			\draw[->,color=cyan] (1)  edge[loop above] node[midway, above]{$-n+1$}  (1) ;
			\draw[->,color=red] (1)  edge[loop below] node[midway, below]{$n-1$}  (1) ;

			\draw[->,color=blue] (3)  edge[bend right] node[midway, above]{$i$}  (2) ;
			\draw[->,color=orange] (3)  edge[bend left] node[midway, below]{$i$}  (2) ;
			
			\draw[->,color=cyan] (2)  edge[loop above] node[midway, above]{$-n+i+1$}  (2) ;
			\draw[->,color=red] (2)  edge[loop below] node[midway, below]{$n-i-1$}  (2) ;
			\draw[->,color=cyan] (3)  edge[loop above] node[midway, above]{$-n+i$}  (3) ;
			\draw[->,color=red] (3)  edge[loop below] node[midway, below]{$n-i$}  (3) ;

			\draw[->,color=blue] (5)  edge[bend right] node[midway, above]{$n$}  (4) ;
			\draw[->,color=orange] (5)  edge[bend left] node[midway, below]{$n$}  (4) ;
			
			\draw[->,color=cyan] (4)  edge[loop above] node[midway, above]{$-1$}  (4) ;
			\draw[->,color=red] (4)  edge[loop below] node[midway, below]{$1$}  (4) ;

		\end{tikzpicture}
	\end{center}
	It is clear from this diagram that the subspaces $V_k=\ab{v_i \mid i \leq k}$ are subrepresentations: no arrows leave these sets, meaning that there are no elements in $V_k$ which get mapped to $V\setminus V_k$.
	
	Indeed, these are the only subrepresentations, which implies that the representation $V$ is reducible for $n\geq 1$, but indecomposable, giving us the following. 
	
	\begin{prop}\label{redindec1}
		The subrepresentations of $V$ are precisely the $V_k$ described above, whence $V$ is reducible and indecomposable for $n\geq 1$.
	\end{prop}
	
	\begin{proof}
%
%
{The proof of this result follows the same argument as in the discussion of \Cref{indecomp_nonsimple_CL21}.}
\end{proof}
	
\end{example}

\begin{example}\label{example_line2}
	A similar representation can be obtained by considering a basis ${v_0,\ldots, v_n}$ of a vector space $V$ of dimension $n+1$ and setting
	\begin{align*}
		\rho(x) v_i = iv_{i-1}, &\quad  \rho(y)v_i=iv_i, \quad \mu(x) v_i = -iv_i, \quad  \mu(y)v_i=iv_{i-1}.
	\end{align*}
	The diagram for this representation is as follows
	\begin{center}
		\begin{tikzpicture}[node distance={25mm},thick,main/.style = {draw, circle,minimum size=1.1cm}] 
			\node[main] (0) {$v_0$}; 
			\node[main] (1) [right of=0] {$v_1$};
			\node[main] (2) [right of=1] {$v_{i-1}$};
			\node[main] (3) [right of=2] {$v_{i}$};
			\node[main] (4) [right of=3] {$v_{n-1}$};
			\node[main] (5) [right of=4] {$v_n$};

			\path (1) -- node{$\ldots$} (2);
			\path (3) -- node{$\ldots$} (4);
			
			\draw[->,color=blue] (1)  edge[bend right] node[midway, above]{$1$}  (0) ;
			\draw[->,color=orange] (1)  edge[bend left] node[midway, below]{$1$}  (0) ;
			
			\draw[->,color=cyan] (5)  edge[loop above] node[midway, above]{$n$}  (5) ;
			\draw[->,color=red] (5)  edge[loop below] node[midway, below]{$-n$}  (5) ;
			\draw[->,color=cyan] (1)  edge[loop above] node[midway, above]{$1$}  (1) ;
			\draw[->,color=red] (1)  edge[loop below] node[midway, below]{$-1$}  (1) ;

			\draw[->,color=blue] (3)  edge[bend right] node[midway, above]{$i$}  (2) ;
			\draw[->,color=orange] (3)  edge[bend left] node[midway, below]{$i$}  (2) ;
			
			\draw[->,color=cyan] (2)  edge[loop above] node[midway, above]{$i-1$}  (2) ;
			\draw[->,color=red] (2)  edge[loop below] node[midway, below]{$-i+1$}  (2) ;
			\draw[->,color=cyan] (3)  edge[loop above] node[midway, above]{$i$}  (3) ;
			\draw[->,color=red] (3)  edge[loop below] node[midway, below]{$-i$}  (3) ;

			\draw[->,color=blue] (5)  edge[bend right] node[midway, above]{$n$}  (4) ;
			\draw[->,color=orange] (5)  edge[bend left] node[midway, below]{$n$}  (4) ;
			
			\draw[->,color=cyan] (4)  edge[loop above] node[midway, above]{$n-1$}  (4) ;
			\draw[->,color=red] (4)  edge[loop below] node[midway, below]{$-n+1$}  (4) ;

		\end{tikzpicture}
	\end{center}
	and this also satisfies \Cref{redindec1}.
	
\end{example}

\begin{example}

	The previous examples can be generalised for any finite rooted tree. Let $T$ be a finite tree with root $v_0$. To abbreviate notation, denote by $|v|$ the distance from the vertex $v$ to $v_0$ (so $d(v,v_0)=|v|$). If $V$ is the vector space formally generated by all the vertices of $T$, we can give it a representation structure which generalises \Cref{example_line1}, by setting
	\begin{align*}
		\rho(x) v = |v|\ell(v), &\quad  \rho(y)v=(-n+|v|)v, \quad \mu(x) v =(n-|v|)v, \quad  \mu(y)v=|v|\ell(v), 
	\end{align*}
	where $n=\max\cb{|v| \text{ for } v \in V(T)}$ is the maximum distance from $v_0$ attained in $T$ and $\ell(v)$ is the (only) vertex which satisfies $|\ell(v)|=|v|-1$, for $v \neq v_0$, and $\ell(v_0)=0$.\footnote{The notation comes from the fact that by looking at the tree horizontally with the root on the left and the tree ``growing'' to the right, $\ell(v)$ is the only neighbour of $v$ exactly one position to the left.}
	
	Again, using the diagrams, this means that any rooted tree can define a representation if it has the following form between two adjacent vertices:
	
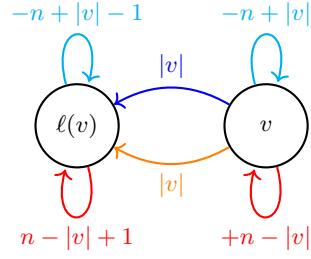
\begin{center}
		\begin{tikzpicture}[node distance={25mm},thick,main/.style = {draw, circle,minimum size=1.1cm}] 
			\node[main] (0) {$\ell(v)$}; 
			\node[main] (1) [right of=0] {$v$};

			\draw[->,color=blue] (1)  edge[bend right] node[midway, above]{$|v|$}  (0) ;
			\draw[->,color=orange] (1)  edge[bend left] node[midway, below]{$|v|$}  (0) ;
			
			\draw[->,color=cyan] (0)  edge[loop above] node[midway, above]{$-n+|v|-1$}  (0) ;
			\draw[->,color=red] (0)  edge[loop below] node[midway, below]{$n-|v|+1$}  (0) ;
			\draw[->,color=cyan] (1)  edge[loop above] node[midway, above]{$-n+|v|$}  (1) ;
			\draw[->,color=red] (1)  edge[loop below] node[midway, below]{$+n-|v|$}  (1) ;
			
		\end{tikzpicture}
	\end{center}
	
	By setting instead 
	\begin{align*}
		\rho(x) v = |v|\ell(v), &\quad  \rho(y)v=|v|v, \quad \mu(x) v =-|v|v, \quad  \mu(y)v=|v|\ell(v), 
	\end{align*}
	one obtains a representation which generalises \Cref{example_line2} to (possibly infinite) rooted trees.

	By using arguments similar to the proof of \Cref{redindec1}, one can see that the subrepresentations of the representations given by a rooted tree $T$ are precisely the subtrees of $T$ with the same root. In this case, it is no longer true that the subrepresentations are contained in one another, but it remains true that there are no complements for any proper nontrivial subrepresentation, as all nontrivial subrepresentations contain the root, and thus, these representations are also reducible and indecomposable.

\end{example}

\subsubsection{Infinite-dimensional representations}

Throughout this subsection, we will assume the base field $\KK$ has characteristic $0$ and identify $\ZZ$ with a subring of $\KK$. We will now explore some infinite-dimensional representations of $CL_{2,4}$.

\begin{rmk}
	Let $\cb{v_n}_{n \in \ZZ}$ be a set of basis vectors indexed by the integers and let $\lambda \in \KK$. We can see that setting
	\begin{align*}
		\rho(x) v_n = (n+\lambda)v_{n-1}, &\quad  \rho(y)v_n=(n+\lambda)v_n, \text{ for all } n \in \ZZ
	\end{align*}
	yields a representation of the non-abelian Lie algebra of dimension $2$ given by generators $x$ and $y$ with product $\sqb{x,y}=x$. This representation has different properties depending on whether the parameter $\lambda$ is an integer. 
\end{rmk}

We now present two different ways of using this representation to construct representations for $CL_{2,4}$.

\begin{example}\label{infinite_rep_same}
	
	We combine two copies of this representation, both going in the same direction (i.e.~with $\rho(x)$ shifting downward in both), to obtain an infinite-dimensional representation $V$ of $CL_{2,4}$ reminiscent of the finite-dimensional ones in \Cref{sectionWeylCounter}. In order for the mixed representation identity to hold, a sign change is required in $\mu(x)$. With this adjustment, the representation is:
	\begin{align*}
		\rho(x) v_n = (n+\lambda)v_{n-1}, &\quad  \rho(y)v_n=(n+\lambda)v_n, \quad \mu(x) v_n = -(n+\lambda)v_{n}, \quad  \mu(y)v_n=(n+\lambda)v_{n-1}.
	\end{align*}
	Graphically, it looks like this:
	
	
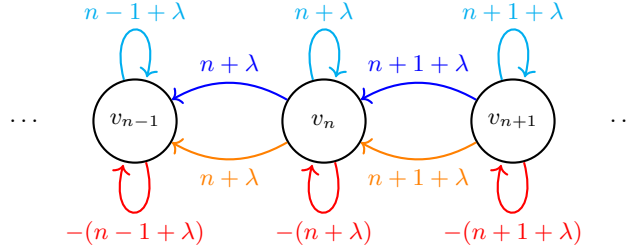
\begin{center}
		\begin{tikzpicture}[node distance={25mm},thick,main/.style = {draw, circle,minimum size=1.1cm}] 
			\node[main] (1) {$v_{n-1}$}; 
			\node[main] (2) [right of=1] {$v_n$};
			\node[main] (3) [right of=2] {$v_{n+1}$};
			\node (4) [right of=3] {};
			\node (0) [left of=1] {};

			\path (3) -- node{$\ldots$} (4);
			\path (1) -- node{$\ldots$} (0);
			
			\draw[->,color=blue] (2)  edge[bend right] node[midway, above]{$n+\lambda$}  (1) ;
			\draw[->,color=orange] (2)  edge[bend left] node[midway, below]{$n+\lambda$}  (1) ;
			
			\draw[->,color=blue] (3)  edge[bend right] node[midway, above]{$n+1+\lambda$}  (2) ;
			\draw[->,color=orange] (3)  edge[bend left] node[midway, below]{$n+1+\lambda$}  (2) ;

			\draw[->,color=cyan] (3)  edge[loop above] node[midway, above]{$n+1+\lambda$}  (3) ;
			\draw[->,color=red] (3)  edge[loop below] node[midway, below]{$-(n+1+\lambda)$}  (3) ;
			
			\draw[->,color=cyan] (1)  edge[loop above] node[midway, above]{$n-1+\lambda$}  (1) ;
			\draw[->,color=red] (1)  edge[loop below] node[midway, below]{$-(n-1+\lambda)$}  (1) ;
			
			\draw[->,color=cyan] (2)  edge[loop above] node[midway, above]{$n+\lambda$}  (2) ;
			\draw[->,color=red] (2)  edge[loop below] node[midway, below]{$-(n+\lambda)$}  (2) ;

		\end{tikzpicture}
	\end{center}
	
\end{example}

	We can divide this family of representations into two cases, depending on whether $\lambda$ is an integer or not, with the resulting families having very distinct properties.
	
	\underline{Case $\lambda \not\in \ZZ$}: In this case, similarly to \Cref{example_line1,example_line2}, the subrepresentations are of the form $V_k=\ab{v_i \mid i \leq k}$, where, in this case, we have a doubly infinite chain of inclusions. 
	
	\begin{prop}\label{redindecinfinite1}
		The subrepresentations of $V$ (as defined in \Cref{infinite_rep_same}) are precisely the $V_k$ described above, whence $V$ is reducible and indecomposable.
	\end{prop}
	
	\begin{proof}
		It is clear that $v_k$ generates $V_k$ by successive applications of $\rho(x)$. We will now show that any linear combination $\sum_{i=j}^{k} a_i v_i$ with $a_j,a_k \neq 0$ and $j \leq k$ can be reduced to $v_k$, and thus generates $V_k$. We will show this by induction on $k-j$.
		
		The base case is when $j=k$, which is trivial. Now assume that any linear combination of the form $\sum_{i=j+1}^{k} a_i v_i$ with $a_{j+1}, a_k \neq 0$ 
		can be reduced (via linear combinations and application of representation maps) to $v_k$ and let $v=\sum_{i=j}^{k} a_i v_i$, with $a_j,a_k \neq 0$ and $k-j \geq 1$. We define 
		\[v':=\frac{\rho(y)v}{j+\lambda}=\frac{1}{j+\lambda}\sum_{i=j}^{k} (i+\lambda)a_i v_i  = a_jv_j + \sum_{i=j+1}^{k} \frac{i+\lambda}{j+\lambda}a_i v_i, \]
		which in turn allows us to define
		\[
		w:= v-v' = \sum_{i=j+1}^{k} a_i\big(1-\frac{i+\lambda}{j+\lambda}\big) v_i.
		\]
	We can apply the induction hypothesis to this vector, since $k \neq j$ implies that $1-\frac{k+\lambda}{j+\lambda}\neq 0$. This concludes the proof, as we have reduced $v$ to $w$ and in turn $w$ to $v_k$.
	\end{proof}

	We now consider the other case.
	
	\underline{Case $\lambda \in \ZZ$}: Without loss of generality, we may consider $\lambda = 0$, in which case, the diagram for the representation is as follows.
	\begin{center}
		\resizebox{0.88\textwidth}{!}
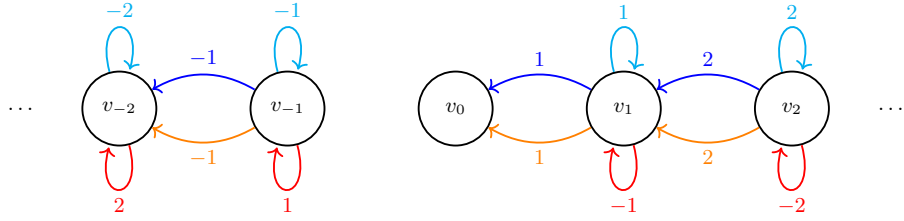
{
			\begin{tikzpicture}[node distance={25mm},thick,main/.style = {draw, circle,minimum size=1.1cm}] 
				\node[main] (0) {$v_{-2}$}; 
				\node[main] (1) [right of=0] {$v_{-1}$};
				\node[main] (2) [right of=1] {$v_{0}$};
				\node[main] (3) [right of=2] {$v_{1}$};
				\node[main] (4) [right of=3] {$v_{2}$};
				\node (5) [right of=4] {};
				\node (6) [left of=0]{};
				
				\path (4) -- node{$\ldots$} (5);
				\path (0) -- node{$\ldots$} (6);

				\draw[->,color=blue] (3)  edge[bend right] node[midway, above]{$1$}  (2) ;
				\draw[->,color=orange] (3)  edge[bend left] node[midway, below]{$1$}  (2) ;
				
				\draw[->,color=cyan] (0)  edge[loop above] node[midway, above]{$-2$}  (0) ;
				\draw[->,color=red] (0)  edge[loop below] node[midway, below]{$2$}  (0) ;
				
				\draw[->,color=cyan] (1)  edge[loop above] node[midway, above]{$-1$}  (1) ;
				\draw[->,color=red] (1)  edge[loop below] node[midway, below]{$1$}  (1) ;

				\draw[->,color=blue] (1)  edge[bend right] node[midway, above]{$-1$}  (0) ;
				\draw[->,color=orange] (1)  edge[bend left] node[midway, below]{$-1$}  (0) ;
				
				\draw[->,color=cyan] (3)  edge[loop above] node[midway, above]{$1$}  (3) ;
				\draw[->,color=red] (3)  edge[loop below] node[midway, below]{$-1$}  (3) ;

				\draw[->,color=blue] (4)  edge[bend right] node[midway, above]{$2$}  (3) ;
				\draw[->,color=orange] (4)  edge[bend left] node[midway, below]{$2$}  (3) ;
				
				\draw[->,color=cyan] (4)  edge[loop above] node[midway, above]{$2$}  (4) ;
				\draw[->,color=red] (4)  edge[loop below] node[midway, below]{$-2$}  (4) ;

		\end{tikzpicture}}
	\end{center}
	Using the same notation as above, we may see that the subspaces $V_k$ are still subrepresentations of $V$, but this time they are not the only ones. In fact, we have that $V^+:=\ab{v_i \mid 0 \leq i}$ and $V_k^+:=\ab{v_i \mid 0 \leq i \leq k}$ are also subrepresentations of $V$. In particular, $V$ can be decomposed into $V=V_{-1}\oplus V^+$, and so $V$ is decomposable.
	
	What remains true is that each component $V_{-1}$ and $V^+$ is indecomposable, the first one by the same argument as the case where $\lambda \not\in \ZZ$ and the second by the argument used in \Cref{redindec1}. In fact, $V^+$ is an infinite-dimensional generalisation of \Cref{example_line2}, and each of its subrepresentations $V_k^+$ is isomorphic to each representation in \Cref{example_line2}.

\begin{example}\label{infinite_rep_opposite}
	We may also ``glue'' representations going ``in opposite ways'', in which case we obtain
	\begin{align*}
		\rho(x) v_n = (n+\lambda)v_{n-1}, &\quad  \rho(y)v_n=(n+\lambda)v_n, \quad \mu(x) v_n = (n+\lambda)v_{n}, \quad  \mu(y)v_n=(n+\lambda)v_{n+1}.
	\end{align*}
	Graphically, it looks like this:
	
	
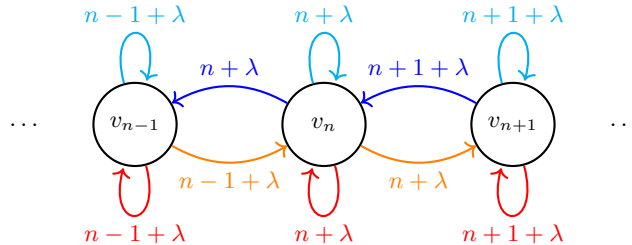
\begin{center}
		\begin{tikzpicture}[node distance={25mm},thick,main/.style = {draw, circle,minimum size=1.1cm}] 
			\node[main] (1) {$v_{n-1}$}; 
			\node[main] (2) [right of=1] {$v_n$};
			\node[main] (3) [right of=2] {$v_{n+1}$};
			\node (4) [right of=3] {};
			\node (0) [left of=1] {};

			\path (3) -- node{$\ldots$} (4);
			\path (1) -- node{$\ldots$} (0);
			
			\draw[->,color=blue] (2)  edge[bend right] node[midway, above]{$n+\lambda$}  (1) ;
			\draw[->,color=orange] (1)  edge[bend right] node[midway, below]{$n-1+\lambda$}  (2) ;
			
			\draw[->,color=blue] (3)  edge[bend right] node[midway, above]{$n+1+\lambda$}  (2) ;
			\draw[->,color=orange] (2)  edge[bend right] node[midway, below]{$n+\lambda$}  (3) ;

			\draw[->,color=cyan] (3)  edge[loop above] node[midway, above]{$n+1+\lambda$}  (3) ;
			\draw[->,color=red] (3)  edge[loop below] node[midway, below]{$n+1+\lambda$}  (3) ;
			
			\draw[->,color=cyan] (1)  edge[loop above] node[midway, above]{$n-1+\lambda$}  (1) ;
			\draw[->,color=red] (1)  edge[loop below] node[midway, below]{$n-1+\lambda$}  (1) ;
			
			\draw[->,color=cyan] (2)  edge[loop above] node[midway, above]{$n+\lambda$}  (2) ;
			\draw[->,color=red] (2)  edge[loop below] node[midway, below]{$n+\lambda$}  (2) ;

		\end{tikzpicture}
	\end{center}
	
	As before, there are two cases to consider, depending on whether $\lambda$ is an integer or not.
	
	\underline{Case $\lambda \not\in \ZZ$}: It is clear that given a basis vector $v_k$ one can generate the whole representation by applying $\rho(x)$ to attain basis vectors $v_j$ with $j<k$ and  $\mu(y)$ to attain basis vectors $v_j$ with $j>k$. This works since $\rho(x)v_j$ and $\mu(y)v_j$ are never zero, as $\lambda \not\in \ZZ$.
	
	Using the proof of \Cref{redindecinfinite1}, one concludes that it is possible to reduce any nonzero linear combination of the basis vectors to a single basis element, whence any nonzero element generates $V$ and we have the following.
	
	\begin{prop}
		The infinite-dimensional representation defined in \Cref{infinite_rep_opposite} is irreducible whenever the parameter $\lambda \not\in \ZZ$.
	\end{prop}

	\underline{Case $\lambda \in \ZZ$}: Again, we may assume in this case that $\lambda =0$. The diagram for the representation is the following.
	\begin{center}
		\resizebox{0.88\textwidth}{!}
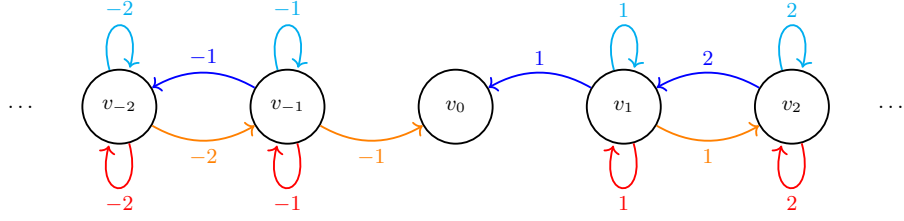
{
			\begin{tikzpicture}[node distance={25mm},thick,main/.style = {draw, circle,minimum size=1.1cm}] 
				\node[main] (0) {$v_{-2}$}; 
				\node[main] (1) [right of=0] {$v_{-1}$};
				\node[main] (2) [right of=1] {$v_{0}$};
				\node[main] (3) [right of=2] {$v_{1}$};
				\node[main] (4) [right of=3] {$v_{2}$};
				\node (5) [right of=4] {};
				\node (6) [left of=0]{};
				
				\path (4) -- node{$\ldots$} (5);
				\path (0) -- node{$\ldots$} (6);

				\draw[->,color=blue] (3)  edge[bend right] node[midway, above]{$1$}  (2) ;
				\draw[->,color=orange] (1)  edge[bend right] node[midway, below]{$-1$}  (2) ;
				
				\draw[->,color=cyan] (0)  edge[loop above] node[midway, above]{$-2$}  (0) ;
				\draw[->,color=red] (0)  edge[loop below] node[midway, below]{$-2$}  (0) ;
				
				\draw[->,color=cyan] (1)  edge[loop above] node[midway, above]{$-1$}  (1) ;
				\draw[->,color=red] (1)  edge[loop below] node[midway, below]{$-1$}  (1) ;

				\draw[->,color=blue] (1)  edge[bend right] node[midway, above]{$-1$}  (0) ;
				\draw[->,color=orange] (0)  edge[bend right] node[midway, below]{$-2$}  (1) ;
				
				\draw[->,color=cyan] (3)  edge[loop above] node[midway, above]{$1$}  (3) ;
				\draw[->,color=red] (3)  edge[loop below] node[midway, below]{$1$}  (3) ;

				\draw[->,color=blue] (4)  edge[bend right] node[midway, above]{$2$}  (3) ;
				\draw[->,color=orange] (3)  edge[bend right] node[midway, below]{$1$}  (4) ;
				
				\draw[->,color=cyan] (4)  edge[loop above] node[midway, above]{$2$}  (4) ;
				\draw[->,color=red] (4)  edge[loop below] node[midway, below]{$2$}  (4) ;

		\end{tikzpicture}}
	\end{center}
	It is clear from the diagram that now there are at least three subrepresentations, namely $\ab{v_0}$, $V^+:=\ab{v_i \mid i \geq 0}$ and $V^-:=\ab{v_i \mid i \leq 0}$.
	
	It can be seen that any basis vector $v_k$ other than $v_0$ generates $V^+$ or $V^-$ depending on whether $k$ is positive or negative. 
	
	Any linear combination of basis vectors $\sum_{i=0}^k a_iv_i$ with $k > 0$ and $a_k \neq 0$ can be reduced to an element of the form $b_0v_0+b_1v_1 $ by applying $\rho(x)^{k-1}$ and to $b_0v_0$ by applying $\rho(x)^{k}$ and possibly scaling, and thus to their difference $b_1v_1$. Thus, any element of the form $\sum_{i=0}^k a_iv_i$ with $k>0$ and $a_k \neq 0$ generates $V^+$.
	
	Similarly, by applying $\mu(y)$, any element of the form $\sum_{i=k}^0 a_iv_i$ with $k<0$ and $a_{k} \neq 0$ generates $V^-$ and any element of the form $\sum_{i=\ell}^k a_iv_i$ with $\ell < 0 < k$ and $a_{k}, a_\ell \neq 0$ generates the whole representation.
	
	We have thus proven the following.
	
	\begin{prop}
		The infinite-dimensional representation $V$ defined in \Cref{infinite_rep_opposite} has exactly $3$ nontrivial subrepresentations whenever $\lambda \in \ZZ$, namely $\ab{v_0}$, $V^+$ and $V^-$ as defined above. Thus, $V$ is reducible and indecomposable.
	\end{prop}
	

\end{example}

\subsubsection{Classification of irreducible finite-dimensional line representations}\label{irredCL2}

\begin{dfn}\label{irred_line_rep}
	We say a finite-dimensional representation of $CL_{2,4}$ is an \emph{irreducible finite-dimensional line representation} if, for some $n \geq 0$, it has a basis ${v_0,\ldots,v_n}$ such that
	\begin{align*}
		\rho(x)v_i = \alpha_i v_{i-1},\quad &\mu(x)v_i=\beta_i v_i, \\ \rho(y)v_i = \gamma_i v_i,\quad &\mu(y)v_i=\delta_i v_{i+1},
	\end{align*}
	for nonzero $\alpha_i$ and $\delta_i$, under the convention that $v_{-1} := 0 =: v_{n+1}$.
\end{dfn}

In other words, an irreducible line representation is one such that $\mu(x)$ and $\rho(y)$ act diagonally and $\rho(x)$ and $\mu(y)$ act nilpotently ``in opposite directions''. The diagram for these kinds of representations is as follows (hence the name ``line representation''):
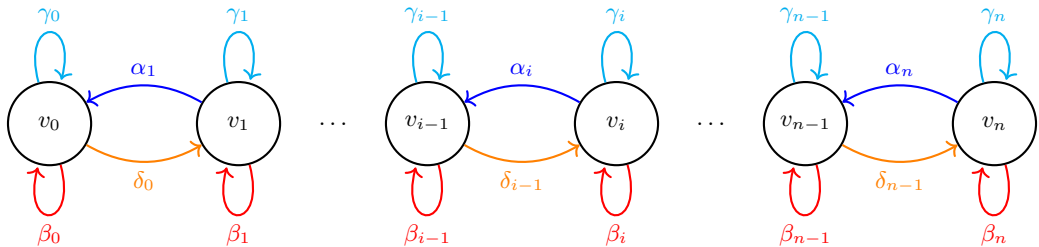
\begin{center}
	\begin{tikzpicture}[node distance={25mm},thick,main/.style = {draw, circle,minimum size=1.1cm}] 
		\node[main] (0) {$v_0$}; 
		\node[main] (1) [right of=0] {$v_1$};
		\node[main] (2) [right of=1] {$v_{i-1}$};
		\node[main] (3) [right of=2] {$v_{i}$};
		\node[main] (4) [right of=3] {$v_{n-1}$};
		\node[main] (5) [right of=4] {$v_n$};

		\path (1) -- node{$\ldots$} (2);
		\path (3) -- node{$\ldots$} (4);
		
		\draw[->,color=blue] (1)  edge[bend right] node[midway, above]{$\alpha_1$}  (0) ;
		\draw[->,color=orange] (0)  edge[bend right] node[midway, below]{$\delta_0$}  (1) ;
		
		\draw[->,color=cyan] (0)  edge[loop above] node[midway, above]{$\gamma_0$}  (0) ;
		\draw[->,color=red] (0)  edge[loop below] node[midway, below]{$\beta_0$}  (0) ;
		\draw[->,color=cyan] (1)  edge[loop above] node[midway, above]{$\gamma_1$}  (1) ;
		\draw[->,color=red] (1)  edge[loop below] node[midway, below]{$\beta_1$}  (1) ;

		\draw[->,color=blue] (3)  edge[bend right] node[midway, above]{$\alpha_i$}  (2) ;
		\draw[->,color=orange] (2)  edge[bend right] node[midway, below]{$\delta_{i-1}$}  (3) ;
		
		\draw[->,color=cyan] (2)  edge[loop above] node[midway, above]{$\gamma_{i-1}$}  (2) ;
		\draw[->,color=red] (2)  edge[loop below] node[midway, below]{$\beta_{i-1}$}  (2) ;
		\draw[->,color=cyan] (3)  edge[loop above] node[midway, above]{$\gamma_i$}  (3) ;
		\draw[->,color=red] (3)  edge[loop below] node[midway, below]{$\beta_i$}  (3) ;

		\draw[->,color=blue] (5)  edge[bend right] node[midway, above]{$\alpha_n$}  (4) ;
		\draw[->,color=orange] (4)  edge[bend right] node[midway, below]{$\delta_{n-1}$}  (5) ;
		
		\draw[->,color=cyan] (4)  edge[loop above] node[midway, above]{$\gamma_{n-1}$}  (4) ;
		\draw[->,color=red] (4)  edge[loop below] node[midway, below]{$\beta_{n-1}$}  (4) ;
		
		\draw[->,color=cyan] (5)  edge[loop above] node[midway, above]{$\gamma_n$}  (5) ;
		\draw[->,color=red] (5)  edge[loop below] node[midway, below]{$\beta_n$}  (5) ;

	\end{tikzpicture}
\end{center}

\begin{rmk}\hfill
	\begin{itemize}
		\item 	We can see that these representations are irreducible, thus justifying their name: for any nonzero element $\sum_i a_i v_i$, we can apply $\rho(x)$ successively, so that we end up with a multiple of $v_0$ and from there apply $\mu(y)$ in order to obtain each of the other basis vectors. This works because all of the $\alpha_i$ and $\delta_i$ are nonzero.
		\item We have previously considered another family of line representations where $\rho(x)$ and $\mu(y)$ act ``in the same direction'', thus justifying the fact we do not just call the above representations ``finite-dimensional line representations''.
		\item Since $\rho(y)$ and $\mu(x)$ act diagonally, they commute. Therefore, combining that with the fact that $\sqb{x,y}=x$ and $\cb{x,y}=y$, in this case, \Cref{rep:id3} reduces to
		\begin{align}\label{rep:id4}
			\rho(y)+\mu(x)&= \rho(x)\mu(y)-\mu(y)\rho(x).
		\end{align}
		\item We do not claim that these are the only finite-dimensional irreducible representations of $CL_{2,4}$; that remains an open question.
	\end{itemize}
	
\end{rmk}

We now proceed by proving some lemmas that will aid us in the final classification result.

\begin{lemma}\label{betagamma}
	In an irreducible finite-dimensional line representation of dimension $n+1$, the coefficients $\beta_i$ and $\gamma_i$ satisfy $\beta_{i+1}= \beta_{i}+1$ and $\gamma_{i+1} = \gamma_{i}+1$, for $0 \leq i \leq n-1$.
\end{lemma}
\begin{proof}
	Using the identity \eqref{rep:id1} applied to the basis vector $v_{i+1}$ one obtains that
	\begin{align*}
		\alpha_{i+1}\gamma_{i+1} v_{i}-\gamma_{i}\alpha_{i+1}v_{i} = \alpha_{i+1} v_{i},
	\end{align*}
	and cancelling the $\alpha_{i+1}$, which is nonzero, we obtain the desired identity for the $\gamma_i$. Using \Cref{rep:id2} we deduce the same for the $\beta_i$.
\end{proof}

\begin{lemma}
	Let $(V,\rho,\mu)$ be an irreducible finite-dimensional line representation with 
	\begin{align*}
		\rho(x)v_i = \alpha_i v_{i-1},\quad &\mu(x)v_i=\beta_i v_i, \quad \rho(y)v_i = \gamma_i v_i, \quad \mu(y)v_i=\delta_i v_{i+1}.
	\end{align*}
	Then, for any $t \in \KK$, $(V,\rho',\mu')$ is also an irreducible finite-dimensional line representation, where
	\begin{align*}
		\rho'(x)=\rho(x),\quad &\mu'(x)v_i=(\beta_i+t) v_i, \quad \rho'(y)v_i = (\gamma_i-t) v_i,\quad \mu'(y)=\mu(y).
	\end{align*}
\end{lemma}

\begin{proof}
	The proof consists of computing the identities \eqref{rep:id1}, \eqref{rep:id2} and \eqref{rep:id4} on the new representation, and observing that the parameter $t$ cancels in each of them.
\end{proof}

\begin{prop}\label{alphadelta}
	An irreducible finite-dimensional line representation satisfies
	\[\beta_i + \gamma_i = -n+2i; \quad \alpha_{i+1}\delta_i = (i+1)(i-n),\]
	for all $0 \leq i \leq n$, using the convention that $\alpha_{n+1} := 0$.
\end{prop}

\begin{proof}
	Applying \Cref{rep:id4} to a basis vector $v_i$, we obtain on the left-hand side
	\begin{align*}
		\rho(y)v_i+\mu(x)v_i = \gamma_i v_i + \beta_i v_i,
	\end{align*}
	and on the right-hand side (with $\delta_{-1} := 0 =: \alpha_{n+1}$)
	\begin{align*}
		\rho(x)\mu(y)v_i-\mu(y)\rho(x)v_i = \rho(x)(\delta_i v_{i+1})-\mu(y)(\alpha_i v_{i-1})=\delta_i\alpha_{i+1}v_i-\alpha_i\delta_{i-1}v_i,
	\end{align*}
	and thus,
	\begin{align}\label{aux_abgd}
		\delta_i\alpha_{i+1}-\alpha_i\delta_{i-1} = \gamma_i + \beta_i.
	\end{align}
	Letting $\theta_i := \beta_i+\gamma_i$, we have that $\theta_i=\theta_0+2i$, from \Cref{betagamma}, and so \Cref{aux_abgd} becomes
	\begin{align}\label{aux2_abgd}
		\delta_i\alpha_{i+1}-\alpha_i\delta_{i-1} = \theta_0+2i.
	\end{align}
	Setting $i=0$ one obtains that $\delta_0\alpha_1=\theta_0$ and by induction it can be seen that
	\begin{align}\label{aux3_abdg}
		\delta_i\alpha_{i+1}=(i+1)\theta_0+i(i+1).
	\end{align}
	On the other hand, setting directly $i=n$ in \Cref{aux2_abgd}, we have
	\begin{align*}
		-\delta_{n-1}\alpha_n=\theta_0+2n.
	\end{align*}
	Comparing this to \Cref{aux3_abdg}, one readily obtains that $\theta_0=-n$ and therefore also that
$\alpha_{i+1}\delta_i = (i+1)(i-n)$, concluding the proof.
\end{proof}

In light of this result, since the $\delta_i$ are fully determined by the $\alpha_i$ and the $\gamma_i$, $\beta_i$ are fully determined by $\beta_0$, we may parametrise each irreducible finite-dimensional line representation as given in \Cref{irred_line_rep} by its dimension $n+1$, the scalar $\beta_0$ and the tuple $\alpha=(\alpha_1, \ldots, \alpha_n)\in~(\KK^\times)^n$. 

Conversely, given such a triple $(n, \beta_0, \alpha)$ and setting $\beta_i = \beta_0 + i$ ($1 \leq i \leq n$), $\gamma_i = i - n - \beta_0$ ($0\leq i \leq n$), and $\delta_i = \alpha_{i+1}^{-1}(i+1)(i-n)$  ($0 \leq i \leq n-1$), the next result shows that we obtain an irreducible finite-dimensional line representation, which we denote by $V(n,\beta_0,\alpha)$.

\begin{theorem}\label{classification_line}
	Using the notation above, we have the following.
	\begin{enumerate}[(a)]
		\item The irreducible finite-dimensional line representations are the $V(n,\beta_0,\alpha)$, in bijection with the triples $(n, \beta_0, \alpha)$ with $n \in \ZZ_{\geq 0}$, $\beta_0 \in \KK$ and $\alpha \in (\KK^\times)^n$.
		\item $V(n,\beta_0,\alpha) \cong V(n',\beta_0',\alpha')$ if and only if $n = n'$ and $\beta_0 = \beta_0'$.
	\end{enumerate}
\end{theorem}
\begin{proof}
	The first statement follows from a direct verification that \Cref{rep:id1,rep:id2,rep:id3} are satisfied for all sets of parameters $(n,\beta_0,\alpha)$. Applying each equation to a basis vector $v_i$ yields a valid equality.
	
	To show that, for fixed $n$ and $\beta_0$, all tuples $\alpha=(\alpha_1, \ldots, \alpha_n)$ yield the same isomorphism class of representations, we note that we may define a new basis $B'=\cb{b_0, \ldots, b_n}$ of $V$ by letting $b_n=v_n$ and $b_i=\alpha_n\alpha_{n-1}\cdots\alpha_{i+1}v_i$ for $i < n$. The parameters relative to this basis are $(n,\beta_0,(1,\ldots,1))$.
	
	It is clear that representations with different dimensions cannot be isomorphic. Finally, different $\beta_0$ yield non-isomorphic representations because the set of eigenvalues for the action of $\mu(x)$ is $\cb{\beta_0,\ldots,\beta_n}$, and is determined by $\beta_0$.
\end{proof}

With this in mind, we may omit the $\alpha$ from the notation for an irreducible finite-dimensional line representation and write $V(n,\beta_0)$ when we do not need to distinguish between isomorphic representations.

\begin{rmk}
	We note that the finite-dimensional line representations we have just classified bear a strong resemblance to the finite-dimensional irreducible representations of the Lie algebra $\mathfrak{sl}_2$. 
	Recall the standard presentation of $\mathfrak{sl}_2$, with basis ${e, f, h}$ and relations
	\[
	\sqb{e,f} = h, \quad \sqb{h,e} = 2e,\quad \sqb{h,f} = 2f.
	\]
	It is well known that for an irreducible $(n+1)$-dimensional representation $\rho'$ of $\mathfrak{sl}_2$, there is a basis ${v_0, v_1, \ldots, v_n}$ where
	\[
	\rho'(e)v_i = (n-i+1) v_{i-1}, \quad \rho'(f)v_i = (i+1)v_{i+1}, \quad \rho'(h)v_i = (n-2i)v_i.
	\]
	Now, letting $\alpha = \nb{n, n-1, \ldots, 2, 1}$, we see that (using the relations in \Cref{alphadelta} in $V(n,\beta_0,\alpha)$)
	\[
		\rho(x)v_i = (n-i+1) v_{i-1}, \quad \mu(y)v_i = - (i+1) v_{i+1}, \quad (\rho(y) + \mu(x)) v_i = -(n-2i)v_i.
	\]
	Thus, on the common underlying vector space $\KK v_0 \oplus \cdots \oplus \KK v_n$, $\rho(x)$ acts as $e \in \mathfrak{sl}_2$, $\mu(y)$ acts as $-f \in \mathfrak{sl}_2$ and $\rho(y)+\mu(x)$ acts as $-h \in \mathfrak{sl}_2$.
\end{rmk}

\section{Tensor products of representations}\label{tensor_sect}

In this section, we explore tensor products of representations of $CL_{2,4}$ of the type considered in the previous section, obtaining a Clebsch--Gordan type formula. 

\begin{dfn}\label{tensorrepdef}
	The tensor product of two representations $(V,\rho,\mu)$ and $(V',\rho',\mu')$ of a compatible Lie algebra $\g$ is the representation defined as $(V\otimes V',\rho\otimes\rho', \mu\otimes\mu')$, where
	\[
	(\rho\otimes\rho')(x)(v\otimes v')=\rho(x)v\otimes v' + v \otimes \rho'(x)v', \quad (\mu\otimes\mu')(x)(v\otimes v')=\mu(x)v\otimes v' + v \otimes \mu'(x)v',
	\]
	for $x \in \g, v \in V, v' \in V'$
\end{dfn}

%
%
%

We start by giving explicit expressions for the tensor product of finite-dimensional line representations. Taking \Cref{classification_line} into account, we may assume that $\alpha=(1,\ldots,1)$.

\begin{dfn}\label{canonical_Vn}
	The line representation of dimension $n+1$ and parameter $\beta$, which we will denote by $V(n,\beta)$, is the $(n+1)$-dimensional vector space with basis ${v_0, \ldots, v_n}$ and actions given by
	\begin{align*}
		\rho(x)v_i = v_{i-1},\quad &\mu(x)v_i=(\beta+i) v_i, \\ \rho(y)v_i = (i-n-\beta) v_i,\quad &\mu(y)v_i=(i+1)(i-n) v_{i+1}.
	\end{align*}
\end{dfn}

The diagram for this representation is as follows.

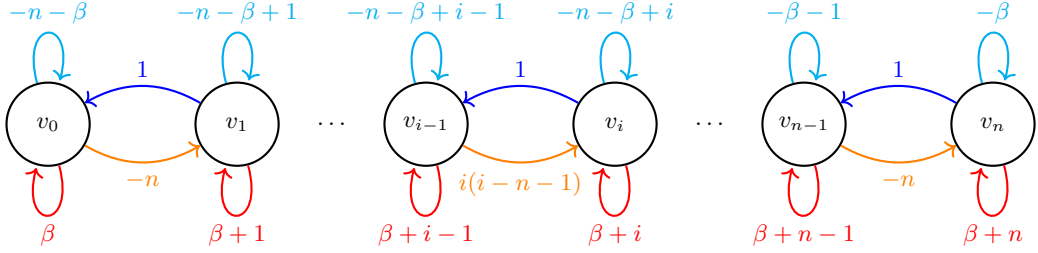
\begin{center}
	\begin{tikzpicture}[node distance={25mm},thick,main/.style = {draw, circle,minimum size=1.1cm}] 
		\node[main] (0) {$v_0$}; 
		\node[main] (1) [right of=0] {$v_1$};
		\node[main] (2) [right of=1] {$v_{i-1}$};
		\node[main] (3) [right of=2] {$v_{i}$};
		\node[main] (4) [right of=3] {$v_{n-1}$};
		\node[main] (5) [right of=4] {$v_n$};

		\path (1) -- node{$\ldots$} (2);
		\path (3) -- node{$\ldots$} (4);
		
		\draw[->,color=blue] (1)  edge[bend right] node[midway, above]{$1$}  (0) ;
		\draw[->,color=orange] (0)  edge[bend right] node[midway, below]{$-n$}  (1) ;
		
		\draw[->,color=cyan] (0)  edge[loop above] node[midway, above]{$-n-\beta$}  (0) ;
		\draw[->,color=red] (0)  edge[loop below] node[midway, below]{$\beta$}  (0) ;
		\draw[->,color=cyan] (1)  edge[loop above] node[midway, above]{$-n-\beta+1$}  (1) ;
		\draw[->,color=red] (1)  edge[loop below] node[midway, below]{$\beta + 1 $}  (1) ;

		\draw[->,color=blue] (3)  edge[bend right] node[midway, above]{$1$}  (2) ;
		\draw[->,color=orange] (2)  edge[bend right] node[midway, below]{$i(i-n-1)$}  (3) ;
		
		\draw[->,color=cyan] (2)  edge[loop above] node[midway, above]{$-n-\beta+i-1$}  (2) ;
		\draw[->,color=red] (2)  edge[loop below] node[midway, below]{$\beta  + i -1$}  (2) ;
		\draw[->,color=cyan] (3)  edge[loop above] node[midway, above]{$-n-\beta + i$}  (3) ;
		\draw[->,color=red] (3)  edge[loop below] node[midway, below]{$\beta + i$}  (3) ;

		\draw[->,color=blue] (5)  edge[bend right] node[midway, above]{$1$}  (4) ;
		\draw[->,color=orange] (4)  edge[bend right] node[midway, below]{$-n$}  (5) ;
		
		\draw[->,color=cyan] (4)  edge[loop above] node[midway, above]{$-\beta-1$}  (4) ;
		\draw[->,color=red] (4)  edge[loop below] node[midway, below]{$\beta + n - 1$}  (4) ;
		
		\draw[->,color=cyan] (5)  edge[loop above] node[midway, above]{$-\beta$}  (5) ;
		\draw[->,color=red] (5)  edge[loop below] node[midway, below]{$\beta + n$}  (5) ;

	\end{tikzpicture}
\end{center}

\begin{nota}
	For simplicity, we will denote $v_i\otimes v_j$ by $w_{ij}$ and $(\rho\otimes\rho)(w_{ij})$, $(\mu\otimes\mu)(w_{ij})$  by $\rho(w_{ij})$, $\mu(w_{ij})$, respectively.
\end{nota}

\begin{prop}[Explicit expressions for the actions on $V(m,\beta)\otimes V(n,\beta')$]
	We have that
	\begin{equation}\label{tensor_id:1}
		\begin{alignedat}{2}
			&	\rho(x)w_{ij} = w_{i-1,j} + w_{i,j-1}, \quad && \rho(y)w_{ij}=(i+j-m-n-\beta-\beta')w_{ij}, \\
			&	\mu(x)w_{ij} = (\beta+\beta'+i+j)w_{ij}, \quad && \mu(y)w_{ij}=(i-m)(i+1)w_{i+1,j} + (j-n)(j+1)w_{i,j+1}.
		\end{alignedat}
	\end{equation}
\end{prop}

\begin{proof}
	The proof is a direct computation that follows from \Cref{tensorrepdef} and is omitted.
\end{proof}

The central result of this section is the following.

\begin{theorem}[Clebsch--Gordan formula for finite-dimensional line representations]
	Assume that the base field $\KK$ has characteristic $0$. Representations of the form $V(m,\beta)\otimes V(n,\beta')$ can be decomposed recursively as follows:
	\[
	V(m,\beta)\otimes V(n,\beta') \simeq V(m+n,\beta+\beta')\oplus \big( V(m-1,\beta+1)\otimes V(n-1,\beta') \big),
	\]
	or explicitly as follows:
	\[
	V(m,\beta)\otimes V(n,\beta') \simeq \bigoplus_{i=0}^m V(m+n-2i, \beta+\beta'+i).
	\]
\end{theorem}

We prove this theorem via three lemmas.

\begin{lemma}\label{subrepD}
	The representation $V(m,\beta)\otimes V(n,\beta')$ has a subrepresentation $D$ isomorphic to $V(m+n,\beta+\beta')$.
\end{lemma}

\begin{lemma}\label{subrepC}
	The representation $V(m,\beta)\otimes V(n,\beta')$ has a subrepresentation $C$ isomorphic to $ V(m-1,\beta+1)\otimes V(n-1,\beta')$.
\end{lemma}

\begin{lemma}\label{DCdirectsum}
	The subrepresentations $D$ and $C$  of $V(m,\beta)\otimes V(n,\beta')$ from the preceding two lemmas are such that
	\[
	V(m,\beta)\otimes V(n,\beta') = D \oplus C.
	\]
\end{lemma}

We will now prove each lemma in turn.


\begin{proof}[Proof of \Cref{subrepD}]
	Take $d_i=\rho(x)^{m+n-i}w_{mn}$, for $i = 0, \ldots, m+n$, so that $d_\ell = \rho(x)d_{\ell+1}$ for $\ell < m+n$.
	Then, we claim that the subspace $D$ of $V(m,\beta)\otimes V(n,\beta')$ generated by the $d_i$ is a subrepresentation isomorphic to $V(m+n,\beta+\beta')$.
	
	The first thing to note is that $d_\ell$ is a linear combination of elements of the form $w_{ij}$ with $i+j=\ell$. This follows from \eqref{tensor_id:1}, by noting that $\rho(x)w_{ij}$ is a linear combination of elements whose indices sum to $i+j-1$ and arguing inductively.		
	
	This immediately shows that $\rho(y)d_\ell = (\ell-(m+n)-(\beta+\beta'))d_\ell$, $\mu(x)d_\ell = (\beta+\beta'+\ell) d_\ell$ and $\rho(x)d_0=0$. So it remains to check that $\mu(y)D \subseteq D$ to prove that $D$ is a subrepresentation. We will proceed by showing that $\mu(y)d_\ell = \lambda_\ell d_{\ell+1}$ for some $\lambda_\ell \in \KK$, which also implies that $D$ is an irreducible line representation of dimension $n+m+1$ and parameter $\beta+\beta'$, hence isomorphic to $V(m+n,\beta+\beta')$.
	
	We prove our claim that $\mu(y)d_\ell = \lambda_\ell d_{\ell+1}$ for some $\lambda_\ell \in \KK$ using downward induction. The base case is $\ell = m+n$, and it is clear that $\mu(y)d_{m+n}=0$.
	Using \eqref{rep:id4}, which is possible since we have already checked that $\mu(x)$ and $\rho(y)$ act diagonally, we have that
	\begin{align*}
		\mu(y)d_\ell &= \mu(y)\rho(x)d_{\ell+1} = \rho(x)\mu(y)d_{\ell+1}-\rho(y)d_{\ell+1}-\mu(x)d_{\ell+1} \\
		&= \rho(x)\lambda_{\ell+1}d_{\ell+2}-(\ell+1-(m+n)-(\beta+\beta'))d_{\ell+1}-(\ell+1+\beta+\beta')d_{\ell}\\
		& = (\lambda_{\ell+1}+m+n-2\ell-2)d_{\ell+1}.
	\end{align*}
 This proves the desired result and, moreover, gives a recurrence relation for the coefficients $\lambda_\ell$.
\end{proof}

\begin{rmk}
	One might use the recurrence relation we obtained, namely
	\[
	\lambda_\ell = \lambda_{\ell+1}+m+n-2\ell-2,
	\]
	to deduce an explicit expression for the $\lambda_\ell$. We arrive precisely at $\lambda_\ell = (\ell+1)(\ell-m-n)$, which is consistent with the expression for the action of $\mu(y)$ on the  irreducible line representation (see \Cref{canonical_Vn}).
\end{rmk}

\begin{proof}[Proof of \Cref{subrepC}]
	We let $C$ be the vector subspace of $V(m,\beta)\otimes V(n,\beta')$ generated by the following vectors
	\[
	c_{ij} = (j+1)w_{i,j+1} - (i+1)w_{i+1,j}, \quad 0 \leq i \leq m-1, \:  0 \leq j \leq n-1.
	\]	
	
	We start by showing that this generating set is linearly independent. By the definition of the $c_{ij}$, we only need to check the following subsets for linear independence
	\[
	C_\ell = \cb{c_{ij}\mid i+j = \ell},
	\]
	for $1\leq \ell \leq n+m-1$. Fix $\ell$ and suppose that the following holds for some coefficients $\zeta_i\in \KK$
	\[
	\sum_{i=0}^\ell \zeta_i c_{i,\ell-i} = 0.
	\]
	In terms of the $w_{ij}$ this is equivalent to 
	\begin{equation}\label{sumswij}
		\sum_{i=0}^\ell \zeta_i \Big( (\ell-i+1)w_{i,\ell-i+1} - (i+1)w_{i+1,\ell-i} \Big) = 0.
	\end{equation}
	The coefficient of $w_{i,\ell-i+1}$ in this expression is 
	\begin{equation}\label{zetacoefs}
		\zeta_i(\ell-i+1)-i\zeta_{i-1}.
	\end{equation}
	Since the sum \eqref{sumswij} evaluates to $0$, each of these coefficients is zero, and we readily show by induction that this implies that each $\zeta_i$ is also zero, for $0\leq i \leq \ell$, proving our claim.
	
%
	We will now proceed by proving that $C$ is a subrepresentation of $V(m,\beta)\otimes V(n,\beta')$ isomorphic to $V(m-1,\beta+1)\otimes V(n-1,\beta')$ by directly computing the actions $\rho(x)$, $\rho(y)$, $\mu(x)$, $\mu(y)$ on the basis elements $c_{ij}$. Whenever indices appear which are out of bounds, we assume the corresponding element is $0$ (namely, $i<0$, $i\geq m$, and $j\geq n$ when indexing elements of type $c_{ij}$ or $i<0$, $i> m$, and $j> n$ when indexing elements of type $w_{ij}$).
	\begin{align*}
		\rho(x)c_{ij} & = \rho(x)\nb{(j+1)w_{i,j+1} - (i+1)w_{i+1,j}} \\
		&= (j+1)w_{i-1,j+1} + (j+1)w_{i,j} - (i+1)w_{i,j} - (i+1)w_{i+1,j-1} \\
		&= (j+1)w_{i-1,j+1}-iw_{i,j} + j w_{i,j} - (i-1)w_{i+1,j-1} \\
		&= c_{i-1,j}+c_{i,j-1};
		\\
		\\
		\rho(y)c_{ij} & = \rho(y)\nb{(j+1)w_{i,j+1} - (i+1)w_{i+1,j}} \\
		&= (i+j+1-m-n-\beta-\beta')(j+1)w_{i,j+1} - (i+1+j-m-n-\beta-\beta')w_{i+1,j} \\
		&= (i+j-(m-1)-(n-1)-(\beta+1)-\beta')c_{ij};
		\\
		\\
		\mu(x)c_{ij} & = \mu(x)\nb{(j+1)w_{i,j+1} - (i+1)w_{i+1,j}}  \\
		&= (\beta+\beta'+i+j+1)(j+1)w_{i,j+1} - (\beta+\beta'+i+1+j)(i+1)w_{i+1,j} \\
		&= ((\beta+1)+\beta'+i+j)c_{ij};
		\\
		\\
		\mu(y)c_{ij} & = \mu(y)\nb{(j+1)w_{i,j+1} - (i+1)w_{i+1,j}} \\
		&= (j+1)\Big( (i-m)(i+1)w_{i+1,j+1} + (j+1-n)(j+2)w_{i,j+2}  \Big) \\
		& \qquad + (i+1)\Big((i+1-m)(i+2)w_{i+2,j} + (j-n)(j+1)w_{i+1,j+1} \Big) \\
		&= (j+1)(j+1-n)(j+2)w_{i,j+2} - (i+1)(i+1-m)(i+2)w_{i+2,j} \\
		& \qquad +(i+1)(j+1)(i-m+j-n)w_{i+1,j+1} \\
		&= (i-m+1)(i+1)\Big( (j+1)w_{i+1,j+1} - (i+2)w_{i+2,j}  \Big) \\
		& \qquad +  (j-n+1)(j+1)\Big( (j+2)w_{i,j+2} - (i+1)w_{i+1,j+1}  \Big) \\
		&= (i+1)(i-(m-1))c_{i+1,j} + (j+1)(j-(n-1))c_{i,j+1}.
	\end{align*}
	Looking at the coefficients in the above expressions and comparing with \eqref{tensor_id:1}, we see that $C$ is a subrepresentation isomorphic to $V(m-1,\beta+1)\otimes V(n-1,\beta')$.
\end{proof}

\begin{proof}[Proof of \Cref{DCdirectsum}]
	A dimension count shows that if $C$ and $D$ have trivial intersection, they satisfy $V(m,\beta)\otimes V(n,\beta') = D \oplus C$, as $\dim V(m,\beta)\otimes V(n,\beta') = (m+1)(n+1)=\dim D + \dim C = (m+n+1) + mn$.
	
	We are left to show that $C$ and $D$ have trivial intersection, and since $D$ is irreducible, it suffices to show that $D \not\subseteq C$. But from the definitions of $C$ and $D$, this is immediate, as $w_{mn} = d_{m+n} \in D$ but $w_{mn} \not\in C$.
\end{proof}

\section*{Acknowledgements}

The first, second and third authors are supported by the projects PID2024-155502NB-I00 granted by MICIU/AEI/10.13039/501100011033; and by Xunta de Galicia through the Competitive Reference Groups (GRC), ED431C 2023/31.

The third author is supported by an FCT -- Funda\c c\~ao para a Ci\^encia e a Tecnologia, I.P scholarship with reference number 2023.00796.BD.

The third and fourth authors are partially supported by CMUP -- Centro de Matem\'atica da Universidade do Porto, member of LASI, which is financed by national funds through FCT -- Funda\c c\~ao para a Ci\^encia e a Tecnologia, I.P., under the project with reference UID/00144/2025, doi: \url{https://doi.org/10.54499/UID/00144/2025}. 


\end{document}